\documentclass[11pt]{amsart}%
\usepackage{amsmath}
\usepackage{amssymb}
\usepackage{amsthm}
\usepackage[mathscr]{eucal}
\usepackage{mathrsfs}
\usepackage{psfrag}
\usepackage{fullpage}
\usepackage{color}
\usepackage{eucal}
\usepackage[dvips]{graphicx}
\usepackage{amsfonts}%
\usepackage{amsaddr}
\usepackage{enumerate}
\providecommand{\U}[1]{\protect\rule{.1in}{.1in}}
\newtheorem{proposition}{Proposition}[section]

\newtheorem{lemma}[proposition]{Lemma}

\newtheorem{definition}[proposition]{Definition}
\newtheorem{remark}[proposition]{Remark}
\newtheorem{example}[proposition]{Example}

\newtheorem{assumption}[proposition]{Assumption}

\newtheorem{prop}[proposition]{Proposition}

\newtheorem{ass}[proposition]{Assumption}

\newtheorem{lem}[proposition]{Lemma}

\newtheorem{rem}[proposition]{Remark}

\numberwithin{equation}{section}
\numberwithin{proposition}{section}

\newcommand{\eps}{\epsilon}
\newcommand{\R}{\mathbb{R}}

\newcommand{\bra}[1]{\left[#1\right]}
\newcommand{\cbra}[1]{\left\{#1\right\}}
\newcommand{\esp}[1]{E\left[#1\right]}
\newcommand{\espalt}[2]{E^{#1}\bra{#2}}
\newcommand{\condespalt}[3]{E^{#1}\bra{#2\ \big|\ #3}}
\newcommand{\var}[1]{\textrm{Var}\bra{#1}}
\newcommand{\varalt}[2]{\textrm{Var}^{#1}\bra{#2}}
\newcommand{\prob}{\mathbb{P}}
\newcommand{\qprob}{\mathbb{Q}}
\newcommand{\filt}{\mathbb{F}}
\newcommand{\filtg}{\mathbb{G}}
\newcommand{\filth}{\mathbb{H}}
\newcommand{\F}{\mathcal{F}}
\newcommand{\G}{\mathcal{G}}
\newcommand{\Hh}{\mathcal{H}}
\newcommand{\relent}[2]{H\left(#1\ \big| \ #2\right)}
\newcommand{\essinf}[2]{\textrm{essinf}_{#1}\left(#2\right)}
\newcommand{\esssup}[2]{\textrm{esssup}_{#1}\left(#2\right)}
\newcommand{\argmax}{\textrm{argmax}}
\newcommand{\argmin}{\textrm{argmin}}
\newcommand{\tm}{\tilde{\mathcal{M}}}

\newcommand{\N}{\mathbb{N}}

\newcommand{\basisn}{\left(\Omega^n,\F^n,\filt^n, \prob^n\right)}
\newcommand{\probspace}{\left(\Omega,\F,\prob\right)}
\newcommand{\probspacen}{\left(\Omega^n, \F^n, \prob^n\right)}
\newcommand{\reals}{\mathbb{R}}
\newcommand{\nada}[1]{}

\numberwithin{equation}{section}
\numberwithin{proposition}{section}

\begin{document}

\title{Indifference pricing for Contingent Claims: Large Deviations Effects}

\author{Scott Robertson}
\address{Department of Mathematical Sciences\\
Carnegie Mellon University\\
Pittsburgh, PA 15213}
\email{scottrob@andrew.cmu.edu}
\thanks{S. Robertson is supported in part by the National Science Foundation
  under grant number DMS-1312419.}

\author{Konstantinos Spiliopoulos}
\address{Department of Mathematics \& Statistics\\
Boston University\\
Boston, MA 02215}
\email{kspiliop@math.bu.edu}
\thanks{K. Spiliopoulos is supported in part by the National Science Foundation under grant number DMS-1312124. The authors would like to thank the two referees for carefully reading the manuscript and for many helpful suggestions.}

\date{\today}

\begin{abstract}

We study utility indifference prices and optimal purchasing quantities for a non-traded contingent claim in an incomplete semi-martingale market with vanishing hedging errors. We make connections with the theory of large deviations. We concentrate on sequences of semi-complete markets where in the $n^{th}$ market, the claim $B_n$ admits the decomposition $B_n = D_n+Y_n$. Here, $D_n$ is replicable by trading in the underlying assets $S_n$, but $Y_n$ is independent of $S_n$. Under broad conditions, we may assume that $Y_n$ vanishes in accordance with a large deviations principle as $n$ grows. In this setting, for an exponential investor, we identify the limit of the average indifference price $p_n(q_n)$, for $q_n$ units of $B_n$, as $n\rightarrow \infty$.  We show that if $|q_n|\rightarrow\infty$, the limiting price typically differs from the price obtained by assuming bounded positions $\sup_n|q_n|<\infty$, and the difference is explicitly identifiable using large deviations theory.  Furthermore, we show that optimal purchase quantities occur at the large deviations scaling, and hence large positions arise endogenously in this setting.
%
\end{abstract}

\maketitle

\section{Introduction}\label{S:Introduction}

The goal of this paper is to study utility based indifference prices and optimal position sizes in incomplete semi-martingale markets with vanishing hedging errors. In particular, we make direct and novel connections between large deviations theory and both optimal positions and indifference prices. Furthermore, as heuristics indicate that large positions arise endogenously with vanishing hedging errors, this paper has the alternate goal of understanding the effects of such positions upon indifference prices. To this end, our main results show that in the presence of vanishing hedging errors, not only do large positions arise endogenously  through optimal purchasing, but they also lead to non-trivial, explicitly identifiable, corrections to the limiting indifference price that one would obtain for bounded positions.

The financial motivation for studying large investors in incomplete markets comes from the observed notational amounts outstanding in complex financial instruments such as derivatives, mortgage backed securities, and both life insurance and mortality contracts. For example, the over-the-counter derivatives market now has more than $700$ trillion notional outstanding (see \cite{BIS_Data}). These products are neither easily traded, nor replicable by investing in an underlying market, so it is natural to study them within the framework of utility based analysis in incomplete markets.

The idea for using utility functions to price random outcomes has a long history, dating back to the 18th century in early works by Bernoulli and Cramer (see \cite{MR2169807} for a brief historical overview). In the modern economics literature, utility based pricing for contingent claims in the absence of an underlying market dates at least back to \cite{pratt1964risk}, and has been used extensively.  In the presence of an underlying market where partial hedging of the claim is possible, the use of utility based pricing is generally first credited to \cite{HN1989}, and in recent years, it has attracted the attention of many authors. For a comprehensive review of utility based pricing in the mathematical finance literature see \cite{MR2547456}.

In the current paper, for a given utility function $U$, the (average bid) indifference price $p(x,q)$ for initial capital $x$ and $q$ units of a contingent claim $B$ is defined through the balance equation
\begin{equation}\label{eq: indiff_px_a}
u(x-qp(x,q),q) = u(x,0).
\end{equation}
Above, for any $x',q'\in\reals$, $u(x',q')$ is the optimal utility an investor with utility function $U$ can achieve, starting with initial capital $x'$ and $q'$ units of $B$, by trading in the underlying market. As is well known,  $p(x,q)$ typically does not admit an explicit formula and hence some approximation is necessary.  Here, we consider the approximation when $q$ is large and when the hedging error associated with $B$ is small.

That position size is intimately connected to hedging error comes from the simple observation that in a complete market, where there is only one fair price $d$ for a given claim, if one is able to purchase claims for a price $\tilde{p}\neq d$ then it is optimal to take an \emph{infinite} position.  Clearly, complete markets are an idealization of reality, and for practical purposes one cannot take an infinite position.  However, this idea indicates that for small hedging errors, large positions should arise endogenously.  Indeed, this is the underlying motivation for the indifference price approximations in the basis risk models of \cite{davis1997opi,MR1926237, Robertson_2012}, where the traded and non-traded assets are closely correlated.

In a general incomplete market with underlying tradable assets $S$, it is difficult to precisely define the ``hedging error'' for the claim $B$. However, such a definition is possible for semi-complete markets. Here, $B$ admits the decomposition $B=D+Y$, where $D$ is replicable by trading in $S$, but $Y$ is independent of $S$. As such, $Y$ represents the hedging error associated with $B$. Semi-complete models were introduced in \cite{MR2011941}, and have been successfully used in settings ranging from the valuation of stop loss contracts (\cite{MR1968945}), to pricing derivatives in energy markets (\cite{Benedetti2015}). To consider the case of vanishing hedging errors, in this paper we embed the semi-complete market into a sequence of semi-complete markets, indexed by $n$, and study the behavior as $n\rightarrow\infty$ assuming the unhedgeable component $Y_n$ vanishes.

An important example of a semi-complete market with small hedging error is a large financial market, where a sequence of risky assets is theoretically available for trading, but, for practical purposes, one must trade in the first $n$ assets. For each $n$ the market is incomplete, as contingent claims may depend upon all the sources of uncertainty, but the ability to hedge improves as $n$ increases. The notion of large markets was introduced in \cite{MR1348197} and since then, several papers have studied theoretical questions related to asymptotic arbitrage and extending the fundamental theorems of asset pricing, see \cite{MR1806101,MR1687136, MR1785165}. In applications, these types of models frequently appear in the insurance industry, see \cite{BEM_2012,MR2178505,Brennan_Schwartz_79,Milevsky2001299,MR2187311} amongst many others. Therefore, after introducing the abstract semi-complete setting in Section \ref{S:SC}, in Section \ref{S:LM_Ex} we present in detail a large market example used in the insurance industry, where assets are geometric Brownian motions (with arbitrary correlations) and where the claim is the sum of independent components.

With these examples and definitions as starting points, in this paper we seek, for an exponential investor, to identify limiting indifference prices and optimal purchase quantities in a sequence of semi-complete markets with vanishing unhedgeable components. Indifference prices are defined as in \eqref{eq: indiff_px_a} and optimal purchase quantities are defined as in \cite{MR2212897, MR3292128} where, for a given arbitrage free price $\tilde{p}$ and initial capital $x$, the optimal position $\hat{q}(x,\tilde{p})$ is that which maximizes $u(x-q\tilde{p},q)$ over all $q$.  In other words, assuming the investor can initially make a one-time purchase of an arbitrary amount of $B$ for a fixed unit price of $\tilde{p}$, $\hat{q}(x,\tilde{p})$ is the optimal amount of $B$ to purchase, taking into account both her preferences and the fact that for a given purchase amount $q$ (which reduces her initial capital by $q\tilde{p}$) she is able to trade in the underlying market.

The novelty of this current work stems from the fact that we recognize a natural and deep relationship between large deviations theory (see \cite{MR1619036,MR722136} for classical manuscripts) and the optimal investment problem for an exponential investor. However, in many cases, non-trivial refinements of the standard results in the large deviations literature are needed. To help motivate our results, we now briefly outline the main argument.  First, for a fixed semi-complete market where $B=D+Y$, Proposition \ref{prop: opt_invest_no_n} below proves for an exponential investor with risk aversion $a>0$ that the indifference price $p(x,q)$ is independent of $x$ and, writing $p(q)$ for $p(x,q)$, satisfies
\begin{equation*}
p(q) = d + \hat{p}(q);\qquad \hat{p}(q)\triangleq - \frac{1}{qa}\log\left(\espalt{\prob}{e^{-qa Y}}\right).
\end{equation*}
In the above, $d$ is the replication cost for the hedgeable portion $D$ and $q\hat{p}(q)$ is the certainty equivalent for $q$ units of the unhedgeable portion $Y$. Thus, in a sequence of semi-complete markets, where $B_n = D_n + Y_n$, the indifference price for $q_n$ units of $B_n$ is
\begin{equation*}
p^n(q_n) = d^n - \frac{1}{q_n a}\log\left(\espalt{\prob^n}{e^{-q_n a Y_n}}\right).
\end{equation*}
If hedging errors are vanishing, it is natural to assume the laws of $Y_n$ weakly converge to the Dirac measure at $0$.  In this instance, if one ignores position size, then, provided the limit exists, the indifference price converges to $d=\lim_{n\rightarrow\infty} d^n$. However, taking into account the position $q_n$, we see that the limiting behavior of $p^{n}(q_{n})$ depends on the limits of both $d^n$ and $\hat{p}_n(q_n)$. Under very broad conditions, one may assume the laws of $Y_n$ satisfy a large deviations principle (LDP) at scaling $r_n$ with rate function $I$ (uniquely minimized at $0$): sufficient conditions are given in Section \ref{S:ExistenceLDP}. In this instance, the limiting behavior of $p^n(q_n)$ depends on how the position $q_{n}$ scales with the large deviations scaling $r_{n}$. In particular, if $q_{n}\approx l r_n, l\in\reals$ then a non-trivial large deviations effect occurs for the limiting indifference price. Indeed, Varadhan's integral lemma (\cite[Section 4.3]{MR1619036}) yields
\begin{equation*}
\lim_{n\rightarrow\infty} p^n(q_n) = d  - \frac{1}{al}\sup_{y}\left(-aly - I(y)\right).
\end{equation*}
Therefore, large positions lead to non-negligible deviations in the limiting indifference price.  Proposition \ref{prop: ldp_indiff_px} makes the above argument precise, but the reason for the deviation is clear: as hedging errors vanish and position sizes increase, an investor has acute sensitivity to the rare events when hedging strategies fail. For exponential investors, the effect of this sensitivity on the indifference price is precisely identifiable through Varadhan's integral lemma, provided that position sizes are in accordance with the large deviation scaling, i.e. $q_{n}\approx l r_n$.

It turns out that large deviations theory also enables us to identify when position sizes, having been obtained optimally, are in accordance with the large deviations scaling. In Proposition \ref{prop: opt_purchase_range} we prove that for all reasonable arbitrage-free prices (see Proposition \ref{prop: opt_purchase_range} for a precise definition) $\tilde{p}^n\neq d^n$, the optimal purchasing quantity $\hat{q}_{n}$ is within the ``large deviations'' regime where $|\hat{q}_n|\approx l r_n, l\in(0,\infty)$. Hence, the large deviations regime is in a sense the natural regime for large investors who purchase optimally. The basic idea for the above result is that for any arbitrage free price $\tilde{p}^n$, the optimal quantity $\hat{q}_n = \hat{q}_n(\tilde{p}^n)$ to purchase is independent of the initial capital $x$ and satisfies the equality
\begin{equation*}
\tilde{p}^n  - d^n =  \frac{\espalt{\prob^n}{Y_n e^{-\hat{q}_n a Y_n}}}{\espalt{\prob^n}{e^{-\hat{q}_na Y_n}}}.
\end{equation*}
Under the LDP assumption, if $|\hat{q}_n|\approx l r_n$ for $l\in[0,\infty]$ then $\lim_{n\rightarrow\infty}(\tilde{p}^n-d^n) = \tilde{p} - d \in \argmax_{y}\left(-aly - I(y)\right)$ (see \cite[Theorem III.17]{MR1739680}).  From here, it is easy to see that if $l=0$ then necessarily $\tilde{p}=d$. Absent this case $|\hat{q}_n|/r_n > l > 0$. That $l=\infty$ also cannot happen requires a more technical proof but still holds true for reasonable prices $\tilde{p}^n\rightarrow\tilde{p}$.

The important point from a financial perspective is that for markets considered in this paper (e.g., large financial markets), there are non-trivial, explicitly identifiable effects on both indifference prices and optimal positions, which arise as $n\rightarrow\infty$, and are explained by the theory of large deviations.

The rest of the paper is organized as follows. In Section \ref{S:SC} we present in detail the general semi-complete framework for a fixed market. We discuss the optimal investment problem, derive the formula for the indifference price, and characterize both the range of arbitrage-free prices and the equation for the optimal purchasing quantity. In Section \ref{S:LM_Ex}, we specialize the results in the setting of large financial markets.  In Section \ref{S:Limit_LDP_no_n} we embed the semi-complete setting into a sequence of markets, and take $n\rightarrow\infty$, making precise connections to large deviations, assuming a LDP holds. Then, in Section \ref{S:ExistenceLDP} we provide general conditions under which the LDP does hold and discuss specific examples.  Lastly, there is an appendix where proofs of several technical results are given.

\section{The Semi-Complete Framework for a Fixed Model}\label{S:SC}

We now present the general semi-complete framework for a fixed model, explicitly identifying the range of arbitrage free prices, as well as both the indifference price and optimal purchase quantity for an exponential investor.  The proofs of all statements made within this section are given in Section \ref{S:SC_Proof} of the Appendix.

As mentioned in the introduction, semi-complete models are those for which the contingent claim $B$ admits the decomposition
\begin{equation}\label{eq: h_decomp}
B = D + Y,
\end{equation}
where $D$ is perfectly replicable by trading in the underlying market, and where $Y$ is ``completely unhedgeable'' in that $Y$ is independent of the underlying assets. To precisely define the semi-complete setup, we impose the following structure on the filtered probability space, assets and claims.  For notational ease, in this section we present results for a fixed semi-complete market.  Then, when considering limiting indifference prices and their connections to large deviations, we embed the semi-complete setup into a sequence of markets.

Let $\probspace$ denote a complete probability space.  We consider a finite time horizon $T$. There is additionally a filtration $\filt$ which admits the decomposition
\begin{assumption}\label{ass: filt}
\begin{equation*}
\filt = \filtg \vee \filth,
\end{equation*}
where $\G_T, \mathcal{H}_T\subset\F$ are $\prob$ independent and where additionally $\filtg$, $\filth$ satisfy the usual conditions, and hence \cite[Theorem 1]{MR671249} so does $\filt$.
\end{assumption}

Assume zero interest rates so that the riskless asset is identically equal to one. As for the risky assets, assume
\begin{assumption}\label{ass: asset_mkt}
With respect to $\filtg$ and $\prob$, $S = (S^1,...,S^d)$ is a d-dimensional, adapted, locally bounded semi-martingale. Furthermore, the $(\prob,\filtg; S)$-market is complete and arbitrage free in that
\begin{enumerate}[(1)]
\item There exists a unique probability measure $\qprob_0$, equivalent to $\prob$ (written $\qprob_0\sim \prob$) on $\G_T$, so that $S$ is a $(\qprob_0,\filtg)-$local martingale and such that the relative entropy $\relent{\qprob_0}{\prob\big |_{\G_T}} = \espalt{\qprob_0}{\log\left(d\qprob_{0}/d\prob\big|_{\G_T}\right)}$ is finite.
\item For every claim $\xi$ which is $\G_T$ measurable and such that $\espalt{\qprob_0}{|\xi|} < \infty$, there exists a unique $x\in\reals$ and $(\prob,\filtg;S)$-integrable (and hence $(\qprob_0,\filtg;S)$-integrable), trading strategy $\Delta = \cbra{\Delta^1,...,\Delta^d}$, where $\Delta^i_t$ denotes the dollars invested in $S^i$ at time $t$, so that $X^{\Delta}_\cdot = x + \int_0^\cdot \Delta_u dS_u$ is a $(\qprob_0,\filtg)$ martingale and such that $X^{\Delta}_T = \xi$, $\prob$-a.s.
\end{enumerate}
\end{assumption}
\begin{remark}\label{rem: cadlag} As the filtration $\filtg$ satisfies the usual conditions, there is a modification of the replicating strategy $X^{\Delta}$ with cadlag paths.  In the sequel, it will thus be assumed that $X^{\Delta}$ is cadlag.
\end{remark}
Regarding the contingent claim $B$, we assume
\begin{assumption}\label{ass: claim_decomp}
$B$ admits the decomposition in \eqref{eq: h_decomp} where $D$ is $\G_T$-measurable and $Y$ is $\Hh_T$ measurable.
\end{assumption}

\subsection{Optimal Investment Problem}\label{SS:opt_invest}

To properly formulate the optimal investment problem, we now state some consequences of Assumptions \ref{ass: filt} and \ref{ass: asset_mkt}.  First of all, we get that $(\prob,\filtg)$-martingales are also $(\prob,\filt)$-martingales and that $S$ is a $(\prob,\filt)$-special semi-martingale. Thus, it makes sense to integrate with $(\prob, \filt;S)$-integrable, $\filt$-predictable processes. Moreover, Assumption \ref{ass: filt} and \cite[Proposition 8]{MR580121} imply that $\filtg$-predictable, $(\prob,\filtg;S)$-integrable processes are also $\filt$-predictable, $(\prob,\filt;S)$-integrable processes and that the stochastic integrals $\int_0^\cdot \Delta_udS_u$ coincide under $\filtg,\filt$.

Now, consider an investor with the exponential utility function $U(x) = -(1/a)e^{-a x}, x\in\R$ where $a > 0$ is the absolute risk aversion. In order to define the class of allowable trading strategies it is first necessary to define the dual class of local martingale measures. To this end define
\begin{equation}\label{eq: mgle_meas_no_n_class}
\mathcal{M} = \cbra{ \qprob \sim \prob \textrm{ on } \F_T : S \textrm{ is a } (\qprob,\filt) \textrm{-local martingale} }.
\end{equation}

Note that we are working with $\qprob\sim\prob$ (and not just $\qprob\ll\prob$): that we can do this follows from Proposition \ref{prop: opt_invest_no_n} below which explicitly shows that the dual optimal measure is equivalent to $\prob$. For exponential utility, the subset of $\mathcal{M}$ with finite relative entropy with respect to $\prob$ plays an important role.  Thus, define
\begin{equation}\label{eq: mgle_meas_no_n_claa_re}
\tm = \cbra{\qprob\in \mathcal{M} : \relent{\qprob}{\prob} < \infty}.
\end{equation}

As Lemma \ref{lem: mart_meas_structure_2} in Appendix \ref{S:SC_Proof} shows, Assumptions \ref{ass: filt} and \ref{ass: asset_mkt} ensure that $\tm\neq\emptyset$.  As is well known, this fact is intimately related to the lack of arbitrage in the $(\prob,\filt;S)$-market, see \cite[Theorem 8.2.1]{MR2200584}.  Now, recall that a trading strategy is represented by $\Delta = \cbra{\Delta^1,...,\Delta^d}$, where $\Delta^i_t$ denotes the dollars invested in $S^i$ at time $t$. We shall denote
 by  $\mathcal{A}$ the set of $\filt$-allowable trading strategies $\Delta$. In particular, $\Delta$ is allowable if it is $\filt$ predictable,  $(\prob,\filt;S)$-integrable, and if the resultant wealth process $X^{\Delta}$ is a $(\qprob,\filt)$-super-martingale for all $\qprob\in\tm$.

For an initial capital $x$ and position size $q$ in $B$, the value function for the investor is given by
\begin{equation}\label{eq: opt_invest_q}
u(x,q) = \sup_{\Delta\in\mathcal{A}} \esp{U(X^\Delta_T + qB)};\qquad X^{\Delta}_\cdot = x + \int_0^\cdot \Delta_udS_u.
\end{equation}

Before handling the case for general claim sizes $q$, we first identify the value function $u(x,0)$: i.e. without the contingent claim. As for exponential utility $u(x,0)=e^{-ax}u(0,0)$ it suffices to consider $x=0$. To this end we have
\begin{proposition}\label{prop: no_claim}
Let Assumptions \ref{ass: filt} and \ref{ass: asset_mkt} hold.  Then there exists an optimal $\Psi\in \mathcal{A}$ to the optimization problem in \eqref{eq: opt_invest_q} for $q=x=0$. In fact, $\Psi$ is $\filtg$-predictable, $(\prob;\filtg;S)$-integrable and satisfies the first order conditions
\begin{equation}\label{eq: qprob_0_ident}
\frac{d\qprob_0}{d\prob}\bigg|_{\G_T} = \frac{e^{-aX^{\Psi}_T}}{\espalt{}{e^{-aX^\Psi_T}}}.
\end{equation}
Lastly, $X^\Psi$ is a $\qprob$-uniformly integrable $(\qprob,\filt)$-martingale for all $\qprob\in\tm$.
\end{proposition}

\subsection{Identification of indifference prices}

For the exponential utility $U(x) = -(1/a)e^{-ax}$, it is well known that the indifference price from \eqref{eq: indiff_px_a} does not depend upon the initial capital and hence we write $p(q)$ for the price and, as above, consider $x=0$ throughout. Furthermore, $u(-qp(q),q) = e^{aq p(q)}u(0,q)$, and hence \eqref{eq: indiff_px_a} implies
\begin{equation}\label{eq: indif_val_funct_no_n}
p(q) = -\frac{1}{aq}\log\left(\frac{u(0,q)}{u(0,0)}\right).
\end{equation}
We now identify the value function $u(0,q)$, along with the optimal trading strategy and optimal local martingale measure.  Heuristically, as $D$ is replicable with some initial capital $d$, it should follow that $p(q) = d + p(q;Y)$ where $p(q;Y)$ is the indifference price for $q$ units of $Y$.  As $Y$ is independent of $S$ it should then follow that trading in $S$ does not matter and hence $p(q;Y)$ coincides with the (average) certainty equivalent $\hat{p}(q;Y)$ for $Y$, defined through
\[
U(0) = \espalt{}{U(q Y-q\hat{p}(q;Y))},
\]
 which takes the form $\hat{p}(q,Y) = -1/(qa)\log\left(\espalt{}{e^{-qaY}}\right)$.  To make this argument precise one must overcome the fact that even though $S$ is independent of $Y$, the admissible trading strategies $\Delta$ need only be $\filt$-predictable and hence are not independent of $Y$.  However, as will be seen, this essentially makes no difference.

Define the cumulant generating function $\Lambda$ for $Y$ by
\begin{equation}\label{eq: h2_cgf}
\Lambda(\lambda) = \log\left(\espalt{}{e^{\lambda Y}}\right);\qquad \lambda\in\reals.
\end{equation}

In view of Assumption \ref{ass: asset_mkt} and the formula for $\hat{p}(q;Y)$, we impose the following natural integrability conditions upon $D$ and $Y$
\begin{assumption}\label{ass: abstract_int} For some $\epsilon > 0$, $\espalt{\qprob_0}{|D|^{1+\eps}} < \infty$.  For all $\lambda\in\reals$, $\Lambda(\lambda) < \infty$.
\end{assumption}

\begin{remark} Note that Assumption \ref{ass: abstract_int} does not require $B$ to be bounded, but does require some integrability. In particular, the unhedgeable component $Y$ is assumed to have exponential moments of all orders and the slightly stronger integrability condition required of $D$, as opposed to that in Assumption \ref{ass: asset_mkt}, is needed to ensure that $(\qprob_0,\filtg)$-martingales are $(\qprob,\filt)$-martingales as well, for any $\qprob\in\tm$.
\end{remark}

Under Assumption \ref{ass: abstract_int} we have the following result
\begin{proposition}\label{prop: opt_invest_no_n}
Let Assumptions \ref{ass: filt}, \ref{ass: asset_mkt}, \ref{ass: claim_decomp} and \ref{ass: abstract_int} hold. Then, for each $q\in\reals$:
\begin{equation}\label{eq: u_no_n_vf}
\frac{u(0,q)}{u(0,0)} = e^{-qa d}\espalt{}{e^{-qa Y}},
\end{equation}
where $d$ is the initial capital required to replicate $D$. Thus, the indifference price $p$ takes the form
\begin{equation}\label{eq: indif_price_no_n}
\begin{split}
p(q) &= d - \frac{1}{qa}\log\left(\espalt{}{e^{qaY}}\right) = d - \frac{1}{qa}\Lambda(-qa).
\end{split}
\end{equation}
The $\filt$-optimal (in fact $\filtg$-predictable, $(\prob,\filtg;S)$-integrable) trading strategy $\hat{\Delta}\in\mathcal{A}$ is given by $\hat{\Delta} = -q\Delta_1 + \Psi$ where $\Delta_1$ is the replicating strategy for $D$ and $\Psi$ is from Proposition \ref{prop: no_claim}. The resultant wealth process $X^{\hat{\Delta}}$ is a $(\qprob,\filt)$-martingale for all $\qprob\in\tm$. Lastly, the optimal local martingale measure $\hat{\qprob}\in\tm$ takes the form
\begin{equation}\label{eq: opt_meas_no_n}
\frac{d\hat{\qprob}}{d\prob} = \frac{d\qprob_0}{d\prob}\bigg|_{\G_T} \frac{e^{-qaY}}{\espalt{}{e^{-qa Y}}}.
\end{equation}

\end{proposition}

\subsection{Arbitrage Free Prices}\label{SS:ArbitrageFreePrices_no_n}

To connect limiting indifference prices to optimal purchase quantities it is of interest to explicitly characterize the range of arbitrage free prices for $B$, as is done in Lemma \ref{lem: arb_free_range_no_n}. Recall that for $Y\neq 0$, the range of arbitrage free prices for $B$ is given by $I = \left(\underline{b}, \bar{b}\right)$ where
\begin{equation}\label{eq: arb_free_no_n_abstract}
\underline{b} = \inf_{\qprob\in\mathcal{M}}\espalt{\qprob}{B};\qquad \bar{b} = \sup_{\qprob\in\mathcal{M}}\espalt{\qprob}{B}.
\end{equation}
In the current setup $\underline{b}$ and $\overline{b}$ can be explicitly identified, as the following lemma shows.
\begin{lemma}\label{lem: arb_free_range_no_n}
Let Assumptions \ref{ass: filt}, \ref{ass: asset_mkt}, \ref{ass: claim_decomp} and \ref{ass: abstract_int} hold. Then
\begin{equation}\label{eq: arb_free_no_n}
\underline{b} = d + \essinf{\prob}{Y};\qquad \bar{b} = d + \esssup{\prob}{Y},
\end{equation}
where $d$ is the initial capital required to replicate $D$. Notice that as $Y$ is not necessarily bounded, $\underline{b}$ and $\bar{b}$ need not be finite.
\end{lemma}

\subsection{Optimal Quantities}\label{SS:OptimalQuantities_no_n}

Let $\tilde{p}\in (\underline{b},\bar{b})$. Assume that at time $0$, the investor may make a one-time purchase of an arbitrary amount of $B$ for a unit price $\tilde{p}$. It is natural to ask what is the optimal number $\hat{q}$ to purchase.  Such a question has been studied in \cite{MR2212897} by solving the problem
\begin{equation*}
\sup_{q\in\R} u(-q\tilde{p},q) = \sup_{q\in\R} e^{a(q\tilde{p}-qp(q))}u(0,0),
\end{equation*}
where the last equality follows by \eqref{eq: indif_val_funct_no_n}, and taking an optimizer $\hat{q}$ if it exists. As $u(0,0) < 0$ we are interested in finding $\hat{q}\in\argmin_{q\in\R}\left(q\tilde{p} - qp(q)\right)$. Using \eqref{eq: indif_price_no_n}, this amounts to solving
\begin{equation}\label{eq: opt_q_var_prob_no_n}
\inf_{q\in\R} \left(q(\tilde{p} - d) + \frac{1}{a}\log\espalt{}{e^{-qa Y}}\right) = \inf_{q\in\R}\left(q(\tilde{p}-d) + \frac{1}{a}\Lambda(-qa)\right),
\end{equation}
and identifying an optimal $\hat{q}$. To this end, we have the following proposition
\begin{proposition}\label{prop: opt_qn_no_n}
Let Assumptions \ref{ass: filt}, \ref{ass: asset_mkt}, \ref{ass: claim_decomp} and \ref{ass: abstract_int} hold. Let $\tilde{p}\in (\underline{b},\bar{b})$ where $\underline{b},\bar{b}$ are as  in \eqref{eq: arb_free_no_n_abstract}. Then there exists a unique $\hat{q}\in\R$ solving \eqref{eq: opt_q_var_prob_no_n}. $\hat{q}$ is the unique real number which satisfies the first order conditions
\begin{equation}\label{eq: hat_qn_foc_no_n}
\tilde{p}-d = \dot{\Lambda}(-qa).
\end{equation}
where $\dot{\Lambda}(\cdot)$ is the derivative of $\Lambda(\cdot)$.
\end{proposition}

\section{A ``Large Market'' Example}\label{S:LM_Ex}

We now present in detail an important example of a semi-complete market.  In fact, the example below constructs a sequence of semi-complete markets and motivates the desire to study semi-complete markets with asymptotically vanishing hedging errors. Proofs of all the statements within this Section are given in Section \ref{S: LM_Ex_Proof} in the Appendix.

Fix an integer $n$. The large market example refers to a market in which  a sequence of risky assets is in theory available to trade, but  in practice it is only feasible to trade in the first $n$ assets. Contingent claims, however, are dependent upon all the sources of uncertainty and hence for each $n$ the market is incomplete. The semi-complete structure arises when one is able to completely hedge away the portion of the claim depending upon the first $n$ assets or sources of uncertainty, but is unable to hedge the remaining tail portion. As mentioned in the introduction, these types of models typically appear in the insurance industry, see \cite{BEM_2012,MR2178505,Brennan_Schwartz_79, Milevsky2001299,MR2187311}.

The probability space $\probspace$ is fixed and assumed rich enough to support a sequence $\{W^{j}\}_{j\in\mathbb{N}}$ of independent Brownian motions. Denote by $\filt$ the right-continuous, $\prob$-augmented enlargement of the filtration generated by the $\cbra{W^j}_{j\in\N}$ on $[0,T]$. For a given $\mu = \{\mu_i\}_{i\in\mathbb{N}}$ and $\Sigma = \{\Sigma_{ij}\}_{i,j\in\mathbb{N}}$ assume
\begin{assumption}\label{ass: mu_sig}
$\sum_{i=1}^\infty \mu_i^2 < \infty$.  $\Sigma$ is symmetric in that $\Sigma_{ij} = \Sigma_{ji}$ for $i,j\in\mathbb{N}$, and uniformly elliptic in that there exists a $\underline{\lambda}>0$ such that for all $\xi = \{\xi_i\}_{i\in\mathbb{N}}$ with $|\xi|^2 = \sum_{i=1}^\infty \xi_i^2 < \infty$, one has $\xi'\Sigma\xi = \sum_{i,j=1}^\infty \xi_i \Sigma_{ij}\xi_j \geq \underline{\lambda} |\xi|^2$.
\end{assumption}
Set $\sigma$ as the unique \emph{lower triangular} matrix such that $\sigma\sigma' = \Sigma$. Such a $\sigma$ may be obtained using the recursive formula in the Cholesky factorization (\cite[Chapter 6.6]{burden1997numerical}). The risky assets $\cbra{S^i}_{i\in\mathbb{N}}$ evolve according to
\begin{equation}
\frac{dS^{i}_{t}}{S^{i}_{t}}=\mu_{i}dt+\sum_{j=1}^i \sigma_{ij}dW^{j}_{t};\qquad i\in\N,\label{Eq:TradedAssets}
\end{equation}
so that, with an abuse of notation $dS_t/S_t = \mu dt + \sigma dW_t$. It then follows for all $i,j$ that $S^i$ has instantaneous rate of return equal to $\mu^i$ and $S^i,S^j$ have instantaneous return covariances of $\Sigma_{ij}$.

The $\F_T$-measurable non-traded asset $B$ takes the form
\begin{equation}\label{eq: h_form0}
B=\sum_{i=1}^{\infty}B_{i},
\end{equation}
where $B_{i}$ is a random variable measurable with respect to the $\sigma-$ algebra generated by $W^{i}_T$, and hence the $\cbra{B_i}_{i\in\N}$ are independent under $\prob$. For $i\in\N$  define $\Gamma_i$ as the cumulant generating function of $B_i$:
\begin{equation}\label{eq: nth_cgf}
\Gamma_i(\lambda) = \log\left(\espalt{}{e^{\lambda B_{i}}}\right);\qquad \lambda\in\R.
\end{equation}
In order to make $B$ well defined, as well as to verify the assumptions of Section \ref{S:SC}, we assume
\begin{assumption}\label{ass: h_n_mart}
For $i\in\N$ and all $\lambda\in\R$, $\Gamma_i(\lambda) < \infty$.
\end{assumption}

\begin{assumption}\label{ass: h_cgf}
For all $\lambda\in\R$, the limit
\begin{equation}\label{eq: finite_cgf}
\sum_{i=1}^\infty \Gamma_i(\lambda) = \lim_{N\uparrow\infty} \sum_{i=1}^N \Gamma_i(\lambda),
\end{equation}
exists and is finite in magnitude.
\end{assumption}

Notice that we do not assume that $\sum_{i=1}^\infty|\Gamma_{i}(\lambda)|<\infty$; in particular, $\lim_{N\uparrow\infty} \sum_{i=1}^N \Gamma_i(\lambda)$  may depend on the order of summation. Assumption \ref{ass: h_n_mart} implies that $\espalt{}{B_{i}^{2}} < \infty$ for all $i$ and hence $\dot{\Gamma}_i(0) = \espalt{}{B_{i}}$. Additionally, Assumptions \ref{ass: h_n_mart} and \ref{ass: h_cgf} imply that the claim $B$ in \eqref{eq: h_form0} is well defined, as the following lemma shows.

\begin{lemma}\label{lem: h_well_posed}
Let Assumptions \ref{ass: h_n_mart} and \ref{ass: h_cgf} hold.  Then $\sum_{i=1}^N B_i$ converges $\prob$-almost surely and in $L^2(\prob)$ to a random variable $B$. In particular, the limits $\lim_{N\uparrow\infty}\sum_{i=1}^{N}\espalt{}{B_{i}}$ and $\lim_{N\uparrow\infty}\sum_{i=1}^{N}\var{B_{i}}$ exist and are finite.
\end{lemma}

\begin{remark}
The form for $B$ in \eqref{eq: h_form0} also allows for  $B = \sum_{i=1}^\infty \zeta_i \tilde{B}_i$ where $\tilde{B}_i$ is $\sigma(W^i_T)$ measurable and $\cbra{\zeta_i}_{i\in\mathbb{N}}$ is a sequence of normalizing constants.  This form encompasses the case when $B$ is a suitably weighted sum of component claims, or an aggregated claim, see  for example \cite{BEM_2012,MR2178505,Robertson_2012}.
\end{remark}

Define the market price of risk vector $\theta$ by
\begin{equation}\label{eq: theta_def}
\theta = \sigma^{-1}\mu.
\end{equation}

Observe that $\Sigma_{ii}\geq \underline{\lambda}$ and $\sigma_{ii} = \sqrt{\Sigma_{ii}}$, so $\sigma$ is invertible. Note that as $\sigma$ is lower triangular, $\sigma^{-1}$ is also lower triangular with $\sigma^{-1}_{ii} = 1/\sigma_{ii}$. Furthermore, Assumption \ref{ass: mu_sig} implies that $\theta$ may be defined iteratively by $\theta_1 = \mu_1/\sigma_{11}$ and $\theta_i = (1/\sigma_{ii})\left(\mu_i - \sum_{j=1}^{i-1}\sigma_{ij}\theta_j\right)$ for $i\geq 2$. Indeed, a lengthy induction argument shows that $\theta$ thus defined satisfies $\theta = \sigma^{-1}\mu$. Furthermore, Assumption \ref{ass: mu_sig} implies $\sum_{i=1}^\infty \theta_i^2 = \theta'\theta = \mu'\Sigma^{-1}\mu \leq (1/\underline{\lambda})\mu'\mu < \infty$, and hence one may define the measure $\tilde{\qprob} \sim \prob$ on $\F_T$ by
\begin{equation}\label{eq: n_inf_rn}
\frac{d\tilde{\qprob}}{d\prob} = \mathcal{E}\left(\sum_{i=1}^\infty -\theta_i W^i_\cdot\right)_T,
\end{equation}
where $\mathcal{E}(\cdot)$ is the stochastic exponential. With all the notation in place, the market thus described is semi-complete and satisfies the Assumptions of Section \ref{S:SC} as the following lemma shows

\begin{lem}\label{lem: LM_SC}
Let Assumptions \ref{ass: mu_sig}, \ref{ass: h_n_mart} and \ref{ass: h_cgf} hold. Then, for each $n$, with $S = (S^1,...,S^n)$ denoting the tradable assets, $D = \sum_{i=1}^n B_i$, $Y = \sum_{i=n+1}^\infty B_i$ denoting the claim decomposition, $\filtg$ denoting the (right-continuous, $\prob$-augmented) filtration generated by $W^1,...,W^n$ and $\filth$ denoting the (right-continuous $\prob$-augmented) filtration generated by $W^{n+1},W^{n+2},...$ it follows that Assumptions \ref{ass: filt}, \ref{ass: asset_mkt}, \ref{ass: claim_decomp} and \ref{ass: abstract_int} hold. Therefore, with
\begin{equation}\label{eq: x1_def_lm}
d^n = \sum_{i=1}^n \espalt{\tilde{\qprob}}{B_i},
\end{equation}
the indifference price $p^n(q)$ for $q$ units of $B$ satisfies
\begin{equation}\label{eq: indiff_px_n}
p^n(q) = d^n - \frac{1}{qa}\sum_{i=n+1}^\infty \Gamma_i(-qa);\qquad q\in\reals.
\end{equation}
The range of arbitrage free prices from Lemma \ref{lem: arb_free_range_no_n} takes the form $(\underline{b}_n,\bar{b}_n)$ with
\begin{equation}\label{eq: arb_free_range_n}
\begin{split}
\underline{b}_n &= d^n + \sum_{i=n+1}^\infty \essinf{\prob}{B_i};\qquad \bar{b}_n = d^n + \sum_{i=n+1}^\infty \esssup{\prob}{B_i}.
\end{split}
\end{equation}
Lastly, for any $\tilde{p}^n\in (\underline{b}_n,\bar{b}_n)$ the optimal quantity $\hat{q}_n$ from Proposition \ref{prop: opt_qn_no_n} satisfies the first order conditions
\begin{equation}\label{eq: hat_qn_foc}
\tilde{p}^n - d^n = \sum_{i=n+1}^\infty \dot{\Gamma}_i(-\hat{q}_n a).
\end{equation}
\end{lem}

Consider $d^n$ from \eqref{eq: x1_def_lm} and define
\begin{equation}\label{eq: x1_def}
d = \lim_{n\uparrow\infty} d^n.
\end{equation}
That the limit exists follows from Lemma \ref{lem: h_well_posed} which proves $\lim_{N\uparrow\infty} \sum_{i=1}^N \espalt{}{B_i}$ and  $\lim_{N\uparrow\infty} \sum_{i=1}^N \varalt{}{B_i}$ both exist and are finite. Thus, for any positive integers $n\leq N$
\begin{equation*}
d^N- d^n = \sum_{i=n+1}^N \espalt{\tilde{\qprob}}{B_i} = \sum_{i=n+1}^N \espalt{}{B_i} + \sum_{i=n+1}^N\left(\espalt{\tilde{\qprob}}{B_i} - \espalt{}{B_i}\right),
\end{equation*}
and
\begin{equation*}
\begin{split}
\left|\sum_{i=n+1}^N (\espalt{\tilde{\qprob}}{B_i}-\espalt{}{B_i}) \right|^2 = \left|\espalt{}{\frac{d\tilde{\qprob}}{d\prob}\left(\sum_{i=n+1}^N (B_i-\espalt{}{B_i})\right)}\right|^2 &\leq e^{T\sum_{i=1}^\infty \theta_i^2}\sum_{i=n+1}^N \varalt{}{B_i},
\end{split}
\end{equation*}
where the inequality follows by H\"{o}lder's inequality, $\espalt{}{(d\tilde{\qprob}/d\prob)^{2}}=e^{T\theta'\theta}$ and the independence of $\cbra{B_{i}}_{i\in\mathbb{N}}$.
Therefore, the replicating initial capitals $d^n$ are converging to a unique value $d$.  From Lemma \ref{lem: h_well_posed}, $B\in L^{2}(\prob)$ and $d = \espalt{\tilde{\qprob}}{B}$ is the unique arbitrage free price for $B$ in the $n=\infty$ model where one may trade in all the underlying assets $\cbra{S^i}_{i\in\N}$.

To conclude this section, examples are given where the optimal purchase quantities can be explicitly identified. The purpose of these examples is to highlight how optimal positions may become large as $n\uparrow\infty$.

\begin{example}\label{ex: gaussian}
Let $B_{i}\stackrel{\prob}{\sim} N\left(\gamma_{i},\delta_{i}^{2}\right)$ be normally distributed for each $i$.  Assumption \ref{ass: h_n_mart} is always satisfied, and Assumption \ref{ass: h_cgf} follows if $\sum_{i=1}^\infty|\gamma_i|<\infty$, $\sum_{i=1}^\infty \delta_{i}^{2} < \infty$. Moreover, we have $\underline{b}_n = -\infty$ and $\bar{b}_n = \infty$ for each $n$. Calculation using \eqref{eq: nth_cgf} shows $\Gamma_i(\lambda) = (1/2)\lambda^2\delta_{i}^{2} + \lambda\gamma_{i}$. Thus
\begin{equation}\label{eq: gamma_i_sum_val}
\sum_{i=n+1}^\infty \dot{\Gamma}_i(-qa) = -qa\sum_{i=n+1}^\infty\delta_{i}^{2} + \sum_{i=n+1}^\infty \gamma_i.
\end{equation}
Therefore, for any $\tilde{p}^n\in\R$,  $\hat{q}_n$ from \eqref{eq: hat_qn_foc} takes the form
\begin{equation}\label{eq: gaussian_opt_quant}
\hat{q}_n = \frac{d^n-\tilde{p}^n + \sum_{i=n+1}^\infty \gamma_i}{a\sum_{i=n+1}^\infty \delta_{n}^{2}}.
\end{equation}
 Thus, if $\liminf_{n\uparrow\infty}|\tilde{p}^n-d| > 0$ then $|\hat{q}_n|\rightarrow \infty$ at a rate proportional to $\left(\sum_{i=n+1}^\infty \delta_{i}^{2}\right)^{-1}$.

\end{example}

\begin{example}\label{ex: poisson}
Let $B_{i}\stackrel{\prob}{\sim} \textrm{Poi}(\beta_i)$ be Poisson for each $i$. Assumption \ref{ass: h_n_mart} is clearly satisfied, and Assumption \ref{ass: h_cgf} follows if $\sum_{i=1}^\infty \beta_i < \infty$.  Here, $\underline{b}_n = d^n$ and $\bar{b}_n = \infty$ for each $n$. Calculation shows $\Gamma_i(\lambda) = (e^{\lambda}-1)\beta_i$. Thus
\begin{equation}\label{eq: gamma_i_sum_val_p}
\sum_{i=n+1}^\infty \dot{\Gamma}_i(-qa) = e^{-qa}\sum_{i=n+1}^\infty \beta_i.
\end{equation}
As $d^n\uparrow d$, in order for $\tilde{p}^n > \underline{b}_n$ for all $n$, we take $d\leq \tilde{p}^n < \infty$.
Thus, for any $\tilde{p}^n \geq d > d^n $,  $\hat{q}_n$ from \eqref{eq: hat_qn_foc} must satisfy $\hat{q}_n = - (1/a)\log\left((\tilde{p}^n-d^n)/(\sum_{i=n+1}^\infty \beta_i)\right).$ Here, if $\liminf_{n\uparrow\infty}(\tilde{p}^n-d) > 0$ then $\hat{q}_n\rightarrow -\infty$ at a rate proportional to $-\log\left(\sum_{i=n+1}^\infty \beta_i\right)$.

\end{example}

\section{Limiting Indifference Prices, Optimal Quantities and Large Deviations}\label{S:Limit_LDP_no_n}

We now embed the semi-complete market of Section \ref{S:SC} into a sequence of semi-complete markets and let $n\uparrow\infty$.  The goal is to compute limiting indifference prices and optimal position sizes while making connections with the theory of large deviations for the random variables $Y_{n}$, the completely unhedgeable component of the claim $B_n$ in the $n^{th}$ market.  As will be shown, assuming a LDP for $\cbra{Y_n}_{n\in\N}$ (see Definition \ref{Def:LDP} below), large positions arise endogenously when purchasing optimal quantities, and for large quantities there are non-trivial effects on limiting indifference prices. To keep a concrete example in mind, note that for the large market example of Section \ref{S:LM_Ex}, the embedding corresponds to being able to trade in the first $n$ assets $S^1,...,S^n$. To make the embedding precise in the general case we assume
\begin{assumption}\label{ass: asympt_sc}
For each $n\in\mathbb{N}$ there is a complete filtered probability space $\basisn$ with accompanying sub-filtrations $\filtg^n, \filth^n$, assets $S^n$, probability measures $\qprob^n_0$ and claims $B^n$ so that Assumptions \ref{ass: filt}, \ref{ass: asset_mkt}, \ref{ass: claim_decomp} and \ref{ass: abstract_int} hold.
\end{assumption}

At this point, for the convenience of the reader, we recall the definition of the LDP appropriate for our setup
\begin{definition}\label{Def:LDP}
Let $S$ be a Polish space with Borel sigma-algebra $\mathcal{B}(S)$.  Let $\probspacen$ be a sequence of probability spaces.  We say that a collection of random variables $(\xi_n)_{n\in\N}$ from $\Omega^n$ to $S$ has a LDP with good rate function $I:S\to [0,\infty]$ and scaling $r_n$ if $r_n\rightarrow \infty$ and
\begin{enumerate}[(i)]
\item For each $s\ge 0$, the set $\Phi(s) = \left\{ s\in S: I(s)\le s\right\}$
is a compact subset of $S$; in particular, $I$ is lower semi-continuous.
\item For every open $G\subset S$, $\varliminf_{n\uparrow\infty}(1/r_n)\log\left(\prob^n\bra{ \xi_n\in G}\right) \geq -\inf_{s\in G} I(s)$.
\item For every closed $F\subset S$, $\varlimsup_{n\uparrow\infty}(1/r_n)\log\left(\prob^n\bra{ \xi_n\in F}\right)\leq -\inf_{s\in F} I(s)$.
\end{enumerate}
\end{definition}

In this paper we take $S=\mathbb{R}$ and $\xi_{n}=Y_{n}$. By the lower-semicontinuity of $I$, if additionally $I(0) = 0$ if and only if $y=0$ we see that for all $\eps>0$, $\prob^n\bra{|Y_n|\geq \eps}\rightarrow 0$ so that the laws of $Y_n$ are weakly converging to the Dirac mass at $0$. In other words, the unhedgeable component of the contingent claim is vanishing as $n\uparrow\infty$. To motivate why it is reasonable to assume this, consider again the large market example of Section \ref{S:LM_Ex}. Here, according to Lemma \ref{lem: h_well_posed} the unhedgeable component $Y_n = \sum_{i=n+1}^\infty B_i$ is going to $0$ in $L^2(\prob)$, hence in probability.  The strengthening of this convergence from one in probability to a LDP is natural in view of the G\"{a}rtner-Ellis theorem (see Section \ref{S:ExistenceLDP}) and, in light of Varadhan's integral lemma \cite[Section 4.3]{MR1619036},  works particularly well for identifying limiting indifference prices and optimal purchase quantities for an
exponential investor, as is now discussed.

Under Assumption \ref{ass: asympt_sc}, for $n\in\N$ and $q_n\in\R$, the indifference price $p^n(q_n)$ from Proposition \ref{prop: opt_invest_no_n} in the $n^{th}$ market takes the form:
\begin{equation}\label{eq: p_n_first_alt_not_new_a}
\begin{split}
 p^{n}(q_n)  &= d^n -\frac{1}{q_n a}\log\left(\espalt{\prob^n}{e^{-q_n aY_n}}\right),
\end{split}
\end{equation}
where $d^n = \espalt{\qprob^n_0}{D_n}$. Now, assume that $\cbra{Y_n}_{n\in\N}$ satisfies an LDP as in Definition \ref{Def:LDP} with $I(y)=0\Leftrightarrow y=0$. As $d^n$ is the replication cost for the hedgeable component of the claim, and as the unhedgeable component is vanishing according to the LDP, one would naively expect that $\lim_{n\uparrow\infty}p^n(q_n)=\lim_{n\uparrow\infty}d^n=d$. Indeed, this is the case for bounded positions (i.e. $\sup_n|q_n|<\infty)$) as shown in Proposition \ref{prop: ldp_indiff_px} below.  However for unbounded positions, under appropriate integrability assumptions, we see from \eqref{eq: p_n_first_alt_not_new} that Varadhan's integral lemma implies
\begin{equation*}
\lim_{n\uparrow\infty} \frac{q_n}{r_n}\rightarrow l \neq 0\Longrightarrow \lim_{n\uparrow\infty}\left(p^n(q_n)-d^n\right)= -\frac{1}{l a}\sup_{y\in\R}\left(-l a y - I(y)\right).
\end{equation*}
If, as is common  for the large market example of Section \ref{S:LM_Ex}, we assume additionally that $I$ is strictly convex, one obtains for all $l \neq 0$ that $-(1/(a l))\sup_{y\in\R}\left(-l a y - I(y)\right)\neq 0$, and hence there is a non-zero large deviations effect on the limiting indifference price. Now, there are numerous questions which arise from the above heuristic argument:
\begin{enumerate}[(1)]
\item When does the LDP hold?
\item What are the limiting indifference prices if the LDP does hold? What if $q_n/r_n \rightarrow 0$?  What if $|q_n|/r_n \rightarrow \infty$?
\item When does it follow that $q_n/r_n\rightarrow l$ for some $0 < |l| < \infty$?  What is the relationship between $\hat{q}_n$ and $r_n$ for the optimal quantities $\hat{q}_n$ of Proposition \ref{prop: opt_qn_no_n}?
\end{enumerate}

To address these questions, the analysis is split into two parts. First, in Section \ref{SSS:LargeCLAIM_LDP}, a LDP for $\cbra{Y_n}_{n\in\N}$ is assumed to hold. Then, Proposition \ref{prop: ldp_indiff_px} computes limiting indifference prices, showing how prices change with the limiting value of $q_n/r_n$.  Additionally, in Section \ref{SSS:OptimalQuantities_LDP}, Proposition \ref{prop: opt_purchase_range} compares $\hat{q}_n$ to the large deviations rate $r_n$, where $\hat{q}_n$ is the optimal purchase quantity from Proposition \ref{prop: opt_qn_no_n}. Here it is shown that if one can buy shares of the claim at a price $\tilde{p} \neq d^n$ then for all reasonable prices $\tilde{p}$ (as defined below) it follows that $0 < \liminf_{n\uparrow\infty} |\hat{q}_n|/r_n \leq \limsup_{n\uparrow\infty}|\hat{q}_n|/r_n < \infty$. Therefore, one is typically within the large deviations regime where non-trivial effects to the limiting indifference price take place.  Lastly, in Section \ref{S:ExistenceLDP} we give general sufficient conditions which guarantee a LDP for $\cbra{Y_n}_{n\in\N}$, and explicitly prove the LDP for $\cbra{Y_n}_{n\in\N}$ for two large market examples.

\subsection{Large Claim Analysis Assuming an LDP}\label{SSS:LargeCLAIM_LDP}

Denote by
\begin{equation}\label{eq: Y_n_cgf}
\Lambda_n(\lambda) = \log\left(\espalt{\prob^n}{e^{\lambda Y_n}}\right);\qquad \lambda\in\R,
\end{equation}
so that \eqref{eq: p_n_first_alt_not_new_a} becomes
\begin{equation}\label{eq: p_n_first_alt_not_new}
\begin{split}
 p^{n}(q_n)  & = d^n - \frac{1}{q_n a}\Lambda_n(-q_n a).\\
\end{split}
\end{equation}
Note that Assumption \ref{ass: asympt_sc} implies, by applying Assumption \ref{ass: abstract_int} for each $n$, that $\Lambda_n(\lambda)<\infty$ for all $\lambda\in\R$. Also, note that by Holder's inequality
\begin{equation}\label{eq: p_n_dec}
q\mapsto p^n(q) \textrm{ is decreasing}.
\end{equation}
We begin by assuming a LDP for $\cbra{Y_n}_{n\in\N}$
\begin{assumption}\label{ass: ldp}
The random variables $\{Y_{n}\}_{n\in\N}$  satisfy the LDP with scaling $\cbra{r_n}_{n\in\N}$ and with good rate function $I(y)$.\footnote{Note that we do not necessarily assume convexity of the rate function $I(y)$.} Additionally, $I(y)=0\Leftrightarrow y=0$, and there is a constant $\delta > 0$
such that for $\eps= \pm \delta$
\begin{equation}\label{eq: cgf_limit_int}
\limsup_{n\uparrow\infty} \frac{1}{r_n}\log\left(\espalt{\prob^n}{e^{\eps r_n Y_n}}\right)  = \limsup_{n\uparrow\infty} \frac{1}{r_n}\Lambda_n(\eps r_n) < \infty.
\end{equation}
\end{assumption}

The bound \eqref{eq: cgf_limit_int} is a moment condition imposed in order to guarantee the validity of Varadhan's integral lemma for affine functions, and it is necessary to know the maximal bounds $\eps$ which still yield \eqref{eq: cgf_limit_int}. Thus, define
\begin{equation}\label{eq: cgf_bounds}
\begin{split}
\bar{M}&= \sup\cbra{ M: \ \limsup_{n\uparrow\infty} \frac{1}{r_n}\Lambda_n(M r_n) < \infty};\quad \underline{M}= \inf\cbra{M:\ \limsup_{n\uparrow\infty}\frac{1}{r_n}\Lambda_n(Mr_n) < \infty}.
\end{split}
\end{equation}
Assumption \ref{ass: ldp} implies $\bar{M} \geq \delta$ and $\underline{M}\leq -\delta$. Next, define
\begin{equation}\label{eq: cgf_bounds_A}
\begin{split}
M^*&= \sup\cbra{ M: \ \sup_{y\in\R}\left(My - I(y)\right) < \infty};\quad M_*= \inf\cbra{M:\ \sup_{y\in\R}\left(My - I(y)\right) < \infty}.
\end{split}
\end{equation}
 Under Assumption \ref{ass: ldp}, for $\underline{M} < M < \bar{M}$, Varadhan's integral lemma implies\footnote{Note that the moment condition in Varadhan's lemma,
$\limsup_{n\uparrow\infty} (1/r_n)\Lambda_n(\gamma Mr_n)<\infty, \gamma>1$, holds with $\gamma=\hat{M}/M>1$ for any $\hat{M}\in(M,\bar{M})$.}
\begin{equation}\label{eq: varad_pos}
\lim_{n\uparrow\infty} \frac{1}{r_n}\Lambda_n(Mr_n) = \sup_{y\in\R}\left(My - I(y)\right)<\infty.
\end{equation}
Therefore, we see that
\begin{equation}\label{eq: M_M_bounds}
-\infty \leq M_*\leq \underline{M}\leq -\delta < 0 < \delta \leq \bar{M}\leq M^*\leq \infty\footnote{There are many examples where $M_* < \underline{M}$ and $M^* > \bar{M}$: see \cite[Chapter 4.3]{MR1619036}.}.
\end{equation}

Let Assumption \ref{ass: ldp} hold. For any sequence $\cbra{q_n}_{n\in\N}$  such that $|q_n|\rightarrow \infty$, up to oscillations, there are three different regimes in which to study the limiting indifference prices $p^n(q_n)$
\begin{equation}
\lim_{n\rightarrow\infty}\frac{|q_{n}|}{r_{n}}%
\begin{cases}
=0 & \text{Regime 1}\\
\in(0,\infty) & \text{Regime 2}\\
=\infty & \text{Regime 3}\\
\end{cases}.
\label{Def:ThreePossibleRegimes}%
\end{equation}

Proposition \ref{prop: ldp_indiff_px} below gives a detailed characterization of the limiting indifference prices. As the resultant prices take many forms depending upon  $M_*,\underline{M}, \bar{M}$ and $M^*$, for ease of presentation, we first state results when $-\infty =\underline{M}$ and $\bar{M} = \infty$. Note that this forces $-\infty = M_*$, $M^* = \infty$.  Then, the general result is given. Here, limiting indifference prices are identified in the case where $q_n/r_n \rightarrow l$ for all $l\in [-\infty,\infty]$ except where $l \in [-\underline{M}/a,-M_*/a]$ or $l \in [-M^*/a,-\bar{M}/a]$.

\begin{proposition}\label{prop: ldp_indiff_px_good}
Let Assumptions \ref{ass: asympt_sc} and \ref{ass: ldp} hold.  Furthermore, assume that $-\infty = \underline{M}$, $\bar{M} = \infty$ where $\underline{M},\bar{M}$ are in \eqref{eq: cgf_bounds}. Then
\begin{enumerate}[(1)]
\item (Regime 1) If $\lim_{n\uparrow\infty}|q_n|/r_n = 0$ then $\lim_{n\uparrow\infty} \left(p^n(q_n) - d^n\right) = 0$. In particular, this holds if $\sup_n |q_n| < \infty$.
\item (Regime 2) If $\lim_{n\uparrow\infty} |q_n|/r_n = l \in (0,\infty)$ then
\begin{equation}\label{eq: limit_px_ell_finite2}
\begin{split}
\lim_{n\uparrow\infty} \left(p^n(q_n)-d^n\right) = -\frac{1}{l a}\sup_{y\in\R}\left(-l a y - I(y)\right)\in\R.
\end{split}
\end{equation}
\item (Regime 3).
\begin{enumerate}[1)]
\item If $\lim_{n\uparrow\infty} q_n/r_n = \infty$ then
 $\limsup_{n\uparrow\infty}\left(p^n(q_n)-d^n\right) \leq \inf\cbra{y\ | \ I(y) < \infty}.$
\item If $\lim_{n\uparrow\infty} q_n/r_n = -\infty$ then
 $\liminf_{n\uparrow\infty}\left(p^n(q_n)-d^n\right) \geq \sup\cbra{y\ | \ I(y) < \infty}.$
\end{enumerate}
\end{enumerate}
\end{proposition}

As mentioned above, Proposition \ref{prop: ldp_indiff_px_good} is a direct result of the more general case where it is not assumed that $-\infty = \underline{M}$ and $\bar{M} = \infty$.  The general case is now presented (its proof is in Appendix \ref{Appendix:prop: ldp_indiff_px}).
\begin{proposition}\label{prop: ldp_indiff_px}
Let Assumptions \ref{ass: asympt_sc} and \ref{ass: ldp} hold. Then
\begin{enumerate}[(1)]
\item (Regime 1) If $\lim_{n\uparrow\infty}|q_n|/r_n = 0$ then $\lim_{n\uparrow\infty} \left(p^n(q_n)-d^n\right) = 0$. In particular, this holds if $\sup_{n}|q_n| < \infty$.
\item (Regime 2) If $\lim_{n\uparrow\infty} q_n/r_n = l \in (0,\infty)$ then
\begin{equation}\label{eq: limit_px_ell_finite}
\begin{split}
l \in \left(0,-\frac{\underline{M}}{a}\right) & \Longrightarrow \lim_{n\uparrow\infty} \left(p^n(q_n) - d^n\right) = \inf_{y\in\R}\left(y +\frac{1}{al}I(y)\right) \in (-\infty,0];\\
M_* > -\infty, l > -\frac{M_*}{a} & \Longrightarrow \lim_{n\uparrow\infty} \left(p^n(q_n)-d^n\right) = -\infty.
\end{split}
\end{equation}
 If $\lim_{n\uparrow\infty} q_n/r_n = l \in (-\infty,0)$ then
\begin{equation}\label{eq: limit_px_ell_finite0}
\begin{split}
l \in \left(-\frac{\bar{M}}{a},0\right) & \Longrightarrow \lim_{n\uparrow\infty} \left(p^n(q_n)-d^n\right) = \sup_{y\in\R}\left(y + \frac{1}{l a}I(y)\right) \in [0,\infty);\\
M^*<\infty, l < -\frac{M^*}{a} & \Longrightarrow \lim_{n\uparrow\infty} \left(p^n(q_n)-d^n\right) = \infty.
\end{split}
\end{equation}
\item (Regime 3).
\begin{enumerate}[1)]
\item If $\lim_{n\uparrow\infty} q_n/r_n = \infty$ then
\begin{equation*}
\begin{split}
M_* > -\infty &\Longrightarrow \lim_{n\uparrow\infty} \left(p^n(q_n) - d^n\right) = -\infty;\\
\underline{M} = -\infty &\Longrightarrow \limsup_{n\uparrow\infty}\left(p^n(q_n)-d^n\right) \leq \inf\cbra{y\ | \ I(y) < \infty}.
\end{split}
\end{equation*}
\item If $\lim_{n\uparrow\infty} q_n/r_n = -\infty$ then
\begin{equation*}
\begin{split}
M^*  < \infty &\Longrightarrow \lim_{n\uparrow\infty} \left(p^n(q_n) - d^n\right) = \infty;\\
\bar{M} = \infty &\Longrightarrow \liminf_{n\uparrow\infty}\left(p^n(q_n)-d^n\right) \geq \sup\cbra{y\ | \ I(y) < \infty}.
\end{split}
\end{equation*}
\end{enumerate}
\end{enumerate}
\end{proposition}

\subsection{Optimal Quantities and the Large Deviations Rate}\label{SSS:OptimalQuantities_LDP}

Proposition \ref{prop: ldp_indiff_px} shows that depending upon the limit $\lim_{n\uparrow\infty} q_n/r_n = l$ there is a wide range of possible limiting indifference prices.  The purpose of this section is to characterize $l$ when $q_n$ is obtained by purchasing optimal quantities : i.e. $q_n = \hat{q}_n(\tilde{p}^n)$ from \eqref{eq: hat_qn_foc_no_n} for a given $\tilde{p}^n\in (\underline{b}_n,\bar{b}_n)$  where $\underline{b}_n,\bar{b}_n$ are from Lemma \ref{lem: arb_free_range_no_n} and take the form
\begin{equation*}
\underline{b}_n = d^n + \essinf{\prob^n}{Y_n};\qquad \bar{b}_n = d^n + \esssup{\prob^n}{Y_n}.
\end{equation*}
The main result of this section shows that when purchasing optimal quantities the limits $l = 0,\pm\infty$ cannot happen for all ``reasonable prices'' $\tilde{p}^n\in (\underline{b}_n,\bar{b}_n)$ such that $\liminf_{n\uparrow\infty} |\tilde{p}^n-d^n| > 0$.

For ease of presentation, it is necessary to rule out both one particular trivial case and the case of oscillations in the range of arbitrage free prices.  Regarding the trivial case, consider when there is some subsequence (still labeled $n$) and some limit $l$ such that
\begin{equation*}
l = \lim_{n\uparrow\infty}\esssup{\prob^n}{Y_n} = \lim_{n\uparrow\infty}\essinf{\prob^n}{Y_n}.
\end{equation*}
As $(1/r_n)\log\left(\prob^n\bra{Y_n \geq l + \eps}\right) \rightarrow -\infty$ and $(1/r_n)\log\left(\prob^n\bra{Y_n\leq l - \eps}\right) \rightarrow -\infty$ for all $\eps > 0$ it is clear from Assumption \ref{ass: ldp} that $l = 0$ and $I$ takes the form $I(0) = 0, I(y) = \infty, y\neq 0$.  Furthermore, as for any $q_n$, the indifference price $p^n(q_n)$ is arbitrage free, it trivially follows that $\lim_{n\uparrow\infty} p^n(q_n)-d^n = 0$ for all sequences $\cbra{q_n}_{n\in\N}$, irrespective of its relation to $\cbra{r_n}_{n\in\N}$.  Thus, to rule out this trivial case, and to rule out the case of oscillations where prices may be arbitrage free for one sequence $\cbra{n_k}_{k\in\N}$ tending towards infinity but not for another sequence $\cbra{n_j}_{j\in\N}$ (such cases can be treated separately), we assume

\begin{ass}\label{ass: general_arb_free_range} There exist $l < u$ such that $\lim_{n\uparrow\infty} \essinf{\prob^n}{Y_n} = l$ and $\lim_{n\uparrow\infty} \esssup{\prob^n}{Y_n} = u$. Thus, if $p\in (l,u)$ then for large enough $n$, $p+d^n \in (\underline{b}_n,\bar{b}_n)$ is arbitrage free.
\end{ass}

\begin{remark} That $l\leq 0 \leq u$ follows as $I(y) = 0$ if and only if $y=0$. Also, in the large market example of Section \ref{S:LM_Ex} where $\essinf{\prob}{Y_n} =\sum_{i=n+1}^\infty \essinf{\prob}{B_i}$ and $\esssup{\prob}{Y_n} = \sum_{i=n+1}^\infty \esssup{\prob}{B_i}$, it follows that $l=0$ or $l=\infty$ and $u=0$ or $u=\infty$.  Indeed, as Lemma \ref{lem: h_well_posed} shows that $\sum_{i=1}^\infty \espalt{}{B_i}$ exists and is finite, we have for each $n$ that
\begin{equation*}
\sum_{i=n+1}^\infty \esssup{\prob}{B_i} = \sum_{i=n+1}^\infty (\esssup{\prob}{B_i}-\espalt{}{B_i}) + \sum_{i=n+1}^\infty \espalt{}{B_i}.
\end{equation*}
Thus, if $\sum_{i=1}^\infty (\esssup{\prob}{B_i}-\espalt{}{B_i}) < \infty$ then $u = 0$, else $u=\infty$.  A similar statement holds for $l$.  Therefore, Assumption \ref{ass: general_arb_free_range} requires either $l=0,u=\infty$ or $l=-\infty,u=0$ or $l=-\infty, u=\infty$.
\end{remark}

Now, by ``reasonable'' price it is meant that $\tilde{p}^n$ must be such that, from a large deviations perspective, there is some probability of $Y_n$ taking values either below $\tilde{p}^n-d^n$ or above $\tilde{p}^n-d^n$.  This is enforced by assuming that the rate function $I$ from Assumption \ref{ass: ldp} is not identically infinity either above or below $\tilde{p}^n - d^n$.

We assume that for $n$ large enough, $\tilde{p}^n -d^n=p^{n}\in(l,u)$ and $\lim_{n\rightarrow\infty}p^{n}=p\in(l,u)$. In fact, to streamline the presentation, we further assume $p^{n}\equiv p$ is not changing with $n$\footnote{The results of Proposition \ref{prop: opt_purchase_range} are not changed if $\tilde{p}^n-d^n = p^n\rightarrow p\in (l,u)$}.  We thus have the following result.

\begin{proposition}\label{prop: opt_purchase_range}
Let Assumptions \ref{ass: asympt_sc}, \ref{ass: ldp} and \ref{ass: general_arb_free_range} hold. For a given $\tilde{p}^n\in (\underline{b}_n,\bar{b}_n)$ set $\hat{q}_n = \hat{q}_n(\tilde{p}^n)$ as in Proposition \ref{prop: opt_qn_no_n}. Recall the definition of $(l,u)$ in Assumption \ref{ass: general_arb_free_range}. Then
\begin{enumerate}[(1)]
\item Assume $l < 0$ and let $\tilde{p}^n  = d^n + p$ for $l < p < 0$. Then
\begin{enumerate}[a)]
\item $\liminf_{n\uparrow\infty} \hat{q}_n/r_n > 0$.
\item If there exists a $y < p$ such that $I(y) < \infty$ then $\limsup_{n\uparrow\infty} \hat{q}_n/r_n < \infty$.
\end{enumerate}
\item Assume $u>0$ and let $\tilde{p}^n = d^n + p$ for $0 < p < u$.  Then
\begin{enumerate}[a)]
\item $\limsup_{n\uparrow\infty} \hat{q}_n/r_n < 0$.
\item If there exists a $y>p$ such that $I(y) < \infty$ then $\liminf_{n\uparrow\infty} \hat{q}_n/r_n > -\infty$.
\end{enumerate}
\end{enumerate}

\end{proposition}

\begin{proof}
For $Y_n$  note that \eqref{eq: hat_qn_foc_no_n} takes the form
\begin{equation}\label{eq: p_value}
p = \dot{\Lambda}_n(-\hat{q}_na) = \frac{\espalt{\prob^n}{Y_n e^{-\hat{q}_na Y_n}}}{\espalt{\prob^n}{e^{-\hat{q}_naY_n}}}.
\end{equation}
The convexity of $\Lambda_n$ implies the map $q\mapsto \espalt{\prob^n}{Y_ne^{qY_n}}/\espalt{\prob^n}{e^{qY_n}}$ is increasing in $q$.

(Proof of the Statements in $(1)$) Assume that $\liminf_{n\uparrow\infty} \hat{q}_n/r_n \leq 0$.  Let $\eps > 0$ and take a sub-sequence (still labeled $n$) such that for $n$ large enough we may assume that $a\hat{q}_n \leq \eps r_n$. We then have from \eqref{eq: p_value} that
\begin{equation}\label{eq: p_lb_1}
p \geq \frac{\espalt{\prob^n}{Y_n e^{-\eps r_n Y_n}}}{\espalt{\prob^n}{e^{-\eps r_n Y_n}}}.
\end{equation}
Taking $n\uparrow\infty$  and then $\eps\rightarrow 0$ in the above we have
\begin{equation*}
p\geq \liminf_{\eps\rightarrow 0} \liminf_{n\uparrow\infty}\frac{\espalt{\prob}{Y_n e^{-\eps r_n Y_n}}}{\espalt{\prob}{e^{-\eps r_n Y_n }}} = 0,
\end{equation*}
where the last follows by Lemma \ref{lem: exp_tilt_fact} below (with $\eps$ therein being $-\eps$). This gives that $p\geq 0$, but this is a contradiction as $p<0$. Therefore, $\liminf_{n\uparrow\infty}\hat{q}_n/r_n > 0$. Now, assume that $p<0$ is such that $I(y) < \infty$ for some $y<p$.  Assume, by way of contradiction that $\limsup_{n\uparrow\infty} \hat{q}_n/r_n = \infty$ and take a sub-sequence (still labeled $n$) such that $\lim_{n\uparrow\infty} \hat{q}_n/r_n = \infty$. Recall from \eqref{eq: opt_q_var_prob_no_n} that $\hat{q}_n$ minimizes
\begin{equation*}
qp+\frac{1}{a}\Lambda_n(-qa) = qp + \frac{1}{a}\log\left(\espalt{\prob^n}{e^{-qa Y_n}}\right),
\end{equation*}
over $\R$.  In particular, by taking $q=0$ and noting that $\hat{q}_n/r_n\rightarrow \infty$ implies $\hat{q}_n > 0$ for $n$ large enough we have that
\begin{equation*}
\hat{q}_n p + \frac{1}{a}\log\left(\espalt{\prob^n}{e^{-\hat{q}_n a Y_n}}\right) \leq 0 \Longrightarrow \frac{1}{\hat{q}_na}\log\left(\espalt{\prob^n}{e^{-\hat{q}_n a Y_n}}\right) \leq -p.
\end{equation*}
Holder's inequality implies the map $q\mapsto (1/q)\log\left(\espalt{\prob^n}{e^{-qY_n }}\right)$ is increasing for $q>0$.
Now, let $M>0$ be given.  As $\hat{q}_n/r_n\rightarrow \infty$ for $n$ large enough we may assume that $\hat{q}_na \geq M r_n \geq 0$.  Thus
\begin{equation*}
\frac{1}{Mr_n}\log\left(\espalt{\prob^n}{e^{-M r_n Y_n}}\right) \leq -p.
\end{equation*}
We have assumed that $\cbra{Y_n}_{n\in\N}$ satisfies a LDP with scaling $\cbra{r_n}_{n\in\N}$ and rate function $I$.  Thus, for any $M' < M$ the above inequality implies, from Varadhan's integral lemma and Holder's inequality that
\begin{equation*}
\lim_{n\uparrow\infty} \frac{1}{r_n}\log\left(\espalt{\prob^n}{e^{-M'r_n Y_n}}\right) = \sup_{y\in\R}\left(-M' y - I(y)\right) \leq -M'p.
\end{equation*}
Thus, for any $y\in\R$
\begin{equation*}
-M'y - I(y) \leq -M'p \Longrightarrow -y - \frac{I(y)}{M'} \leq -p.
\end{equation*}
Thus, for $y$ such that $I(y)<\infty$ we have, taking $M'\uparrow\infty$, which is allowed because $M>0$ was arbitrary, that $-y\leq -p$ or $y\geq p$.  Thus, $I(y)<\infty$ implies $y\geq p$ which implies that $I(y) = \infty$ for $y < p$.  But, this is a contradiction as it was assumed $p$ was such that $I(y) < \infty$ for some $y<p$.  Therefore, $\limsup_{n\uparrow\infty} \hat{q}_n/r_n < \infty$.

(Proof of the Statements in (2)): The case of $p>0$ is nearly identical to that for $p<0$.  Assume that $\limsup_{n\uparrow\infty} \hat{q}_n/r_n \geq 0$.  Let $\eps > 0$ and take a sub-sequence (still labeled $n$) such that for $n$ large enough we may assume that $a\hat{q}_n \geq -\eps r_n$. Similarly to \eqref{eq: p_lb_1} we have
\begin{equation*}
p \leq \frac{\espalt{\prob^n}{Y_ne^{\eps r_n Y_n}}}{\espalt{\prob^n}{e^{\eps r_n Y_n}}}.
\end{equation*}
Again, taking $n\uparrow\infty$  and then $\eps\rightarrow 0$ we have
\begin{equation*}
p\leq \limsup_{\eps\rightarrow 0}\limsup_{n\uparrow\infty} \frac{\espalt{\prob^n}{Y_n e^{\eps r_n Y_n}}}{\espalt{\prob^n}{e^{\eps r_n Y_n }}} = 0,
\end{equation*}
where the last equality follows by Lemma \ref{lem: exp_tilt_fact} below. This is a contradiction as $p>0$. Therefore, $\limsup_{n\uparrow\infty}\hat{q}_n/r_n < 0$. Now, assume that $p>0$ is such that $I(y) < \infty$ for some $y>p$.  Assume, by way of contradiction that $\liminf_{n\uparrow\infty} \hat{q_n}/r_n = -\infty$ and take a sub-sequence (still labeled $n$) such that $\lim_{n\uparrow\infty} \hat{q}_n/r_n = -\infty$. As $\hat{q}_n$ minimizes
\begin{equation*}
qp+\frac{1}{a}\Lambda_n(-qa) = qp + \frac{1}{a}\log\left(\espalt{\prob^n}{e^{-qa Y_n}}\right)
\end{equation*}
over $\R$, taking $q=0$ gives (recall that $\hat{q}_n < 0$):
\begin{equation*}
\hat{q}_n p + \frac{1}{a}\log\left(\espalt{\prob^n}{e^{-\hat{q}_n a Y_n}}\right) \leq 0 \Longrightarrow -\frac{1}{\hat{q}_na}\log\left(\espalt{\prob}{e^{-\hat{q}_n a Y_n}}\right) \leq p.
\end{equation*}
The map $q\mapsto (1/q)\log\left(\espalt{\prob^n}{e^{qY_n }}\right)$ is increasing for $q>0$.  Now, let $M>0$ be given.  As $\hat{q}_n/r_n\rightarrow -\infty$ for $n$ large enough we may assume that $-\hat{q}_na \geq M r_n \geq 0$.  Thus
\begin{equation*}
\frac{1}{Mr_n}\log\left(\espalt{\prob^n}{e^{M r_n Y_n}}\right) \leq p.
\end{equation*}
By assumption $\cbra{Y_n}_{n\in\N}$ satisfies a LDP with scaling $\cbra{r_n}_{n\in\N}$ and rate function $I$.  Thus, for any $M' < M$ the above inequality implies, from Varadhan's integral lemma and Holder's inequality that
\begin{equation*}
\lim_{n\uparrow\infty} \frac{1}{r_n}\log\left(\espalt{\prob^n}{e^{M'r_n Y_n}}\right) = \sup_{y\in\R}\left(M' y - I(y)\right) \leq M'p.
\end{equation*}
Thus, for any $y\in\R$,
 $M'y - I(y) \leq M'p \Longrightarrow y - \frac{I(y)}{M'} \leq p.$
Thus, for $y$ such that $I(y)<\infty$ we have, taking $M'\uparrow\infty$, which is allowed because $M>0$ was arbitrary, that $y\leq p$.  Thus, $I(y)<\infty$ implies $y\leq p$ which implies that $I(y) = \infty$ for $y > p$.  But, this is a contradiction as it was assumed $p$ was such that $I(y) < \infty$ for some $y>p$.  Therefore, $\liminf_{n\uparrow\infty} \hat{q}_n/r_n > -\infty$.

\end{proof}

\section{On the Existence of the Large Deviations Principle}\label{S:ExistenceLDP}

The goal of this section is to provide conditions under which Assumption \ref{ass: ldp} holds for the random variables $\cbra{Y_n}_{n\in\mathbb{N}}$. Large deviations theory is a well developed subject (see \cite{MR1619036,MR722136}), and one particularly well known and widely used result for proving existence of an LDP is the G\"{a}rtner-Ellis theorem (\cite[Theorem 2.3.6]{MR1619036}), which we now recall.

Consider $\Lambda_n$ from \eqref{eq: Y_n_cgf}. The G\"{a}rtner-Ellis theorem yields a LDP for $\cbra{Y_n}_{n\in\N}$ (see Definition \ref{Def:LDP} with $S=\mathbb{R}$ and $\xi_n=Y_n$) if there exist a sequence $\{r_{n}\}_{n\in\mathbb{N}}$ with $r_n\rightarrow\infty$ such that
\begin{enumerate}[(i)]
\item The limit $\Gamma(\lambda)= \lim_{n\uparrow\infty} (1/r_n)\Lambda_n(\lambda r_n)$ is well defined for each $\lambda\in (-\infty,\infty]$.
\item $0\in \mathcal{D}^{\circ}_{\Gamma}$, the interior of $\mathcal{D}_{\Gamma}=\left\{\lambda\in\mathbb{R}: \Gamma(\lambda)<\infty\right\}$.
\item $\Gamma(\cdot)$ is lower-semi-continuous in $\mathbb{R}$ and differentiable in $\mathcal{D}^{\circ}_{\Gamma}$.
\item Either $\mathcal{D}_{\Gamma}=\mathbb{R}$ or $\Gamma$ is steep at $\partial \mathcal{D}_{\Gamma}$, i.e. for $\lambda\in \mathcal{D}^{\circ}_{\Gamma}$, $\lim_{\lambda\rightarrow \mu}\left|\dot{\Gamma}(\lambda)\right|=\infty$ for every $\mu\in \partial \mathcal{D}_{\Gamma}$.
\end{enumerate}

Indeed, under the above conditions, an LDP for $\cbra{Y_n}_{n\in\mathbb{N}}$ and scaling $\cbra{r_n}_{n\in\mathbb{N}}$ follows with (good) rate function $I(y)=\sup_{\lambda\in\mathbb{R}}\left(\lambda y-\Gamma(\lambda)\right)$, the Legendre-Fenchel transform of $\Gamma(\lambda)$. Furthermore, as the moment condition \eqref{eq: cgf_limit_int} follows by $(ii)$ above, Assumption \ref{ass: ldp} holds.

However, we stress that the G\"{a}rtner-Ellis theorem is only sufficient, and not necessary, for the LDP to hold. To reinforce this, we present two concrete examples from the large market of Section \ref{S:LM_Ex}. In Subsection \ref{SS:GaussianExample}, the LDP is indeed given via the G\"{a}rtner-Ellis theorem and everything (limiting indifference prices, optimal purchasing rates) can be computed explicitly. By contrast, in Subsection \ref{SS:PoissonExample}, even though the G\"{a}rtner-Ellis theorem cannot be used, a LDP still holds and all the quantities of interest for this paper can be computed explicitly by appealing to the special structure of the example.

\subsection{Example : Gaussian case}\label{SS:GaussianExample}

As in Example \ref{ex: gaussian}, assume $B_i\stackrel{\prob}{\sim} N(\gamma_i,\delta_i^2)$ so that
\begin{equation*}
\begin{split}
Y_n &= \sum_{i=n+1}^\infty B_{i} \stackrel{\prob}{\sim} N\left(\sum_{i=n+1}^\infty \gamma_{i},\sum_{i=n+1}^\infty \delta^{2}_{i}\right).
\end{split}
\end{equation*}
Set $r_{n} =  \left(\sum_{i=n+1}^{\infty}\delta^{2}_{i}\right)^{-1}$. Clearly, $\lim_{n\rightarrow\infty}r_{n}=\infty$, and, for any $\lambda\in\R$:
\begin{equation}\label{eq: gaussian_mgf_val}
\begin{split}
\Gamma(\lambda)&=\lim_{n\uparrow\infty}\frac{1}{r_n}\log\left(\espalt{\prob}{e^{\lambda r_n Y_n}}\right) =\lim_{n\uparrow\infty}\left[ \frac{1}{2}\lambda^2+\lambda \frac{\sum_{i=n+1}^\infty \gamma_{i}}{r_{n}}\right]=\frac{1}{2}\lambda^2.
\end{split}
\end{equation}

Thus, the G\"{a}rtner-Ellis theorem implies $\cbra{Y_n}_{n\in\N}$ satisfies a LDP with rate $r_n$ and good rate function $I(y) = \sup_{\lambda\in\reals}\left(\lambda y-\lambda^2/2\right) = y^2/2$, yielding Assumption \ref{ass: ldp}.

\subsubsection*{The Limiting Indifference Price} Recall from \eqref{eq: x1_def} that $\lim_{n\uparrow\infty}d^n = d$ exists. Next, recall the formula for $p^n(q)$ from \eqref{eq: indiff_px_n} and the explicit formula for $\Gamma_i$ from Example \ref{ex: gaussian}. Putting these together, we obtain
\begin{equation}\label{eq: gaussian_exact_prices}
p^n(q_n) - d^n =  -\frac{1}{q_na}\sum_{i=n+1}^\infty \left(\frac{1}{2}q_n^2a^2\delta_i^2 - q_na\gamma_i\right) = -\frac{1}{2}q_na\sum_{i=n+1}^\infty \delta^2_i + \sum_{i=n+1}^\infty\gamma_i = -a\frac{q_n}{2r_n} + \sum_{i=n+1}^\infty \gamma_i,
\end{equation}
where the last equality uses the definition of $r_n$. Thus, $\lim_{n\uparrow\infty}\left(p^n(q_n)-d^n + a q_n/(2r_n)\right) = 0$, and hence for any subsequence (still labeled $n$) such that $\lim_{n\uparrow\infty}|q_n|/r_n$ exists:
\begin{enumerate}[(1)]
\item (Regime 1) If $\lim_{n\uparrow\infty} |q_n|/r_n = 0$ then $\lim_{n\uparrow\infty} p^n(q_n) = d$.
\item (Regime 2) If $\lim_{n\uparrow\infty} |q_n|/r_n = l\neq 0$ then $\lim_{n\uparrow\infty} p^n(q_n) = d -(1/2)al$.
\item (Regime 3) If $\lim_{n\uparrow\infty} |q_n|/r_n = \infty$ then $\lim_{n\uparrow\infty} p^n(q_n)-d = \pm\infty$ if $q_n/r_n\rightarrow\mp\infty$.
\end{enumerate}
Note that because $\underline{M}=M_* = -\infty$, $\bar{M}=M^*=\infty$ and $I(y) = y^2/2$ these results are entirely consistent with Proposition \ref{prop: ldp_indiff_px_good}.

\subsubsection*{Optimal Quantities} Consider the case where $q_n$ is obtained by purchasing optimal quantities : i.e. $q_n = \hat{q}_n$ from \eqref{eq: gaussian_opt_quant}.  Using the definition of $r_n$ it follows that
\begin{equation}\label{eq: gaussian_hat_q_n_ex}
\hat{q}_n = \frac{r_n}{a}\left(d^n - \tilde{p}^n + \sum_{i=n+1}^\infty\gamma_i\right).
\end{equation}
If $\tilde{p}^n-d^n = p \neq 0$ then $\hat{q}_n/r_n \rightarrow -p/a$. If $\tilde{p}^n = d^n$ then $\hat{q}_n/r_n\rightarrow 0$, even though it is certainly possible for $|\hat{q}_n|\rightarrow \infty$, as can easily be seen from \eqref{eq: gaussian_hat_q_n_ex}.

\begin{remark}
For optimal quantities with $\tilde{p}^n -d^n = p$, as $\hat{q}_n/r_n \rightarrow l = -p/a$ we have $\lim_{n\uparrow\infty} p^n(\hat{q}_n) = d + p/2$.
\end{remark}

\subsection{Example: Poisson case}\label{SS:PoissonExample}

As in Example \ref{ex: poisson} assume $B_{i} \stackrel{\prob}{\sim} \textrm{Poi}(\beta_i)$ so that
\begin{equation}\label{eq: Y_n_poisson_dist}
Y_n = \sum_{i=n+1}^\infty B_{i} \stackrel{\prob}{\sim} \textrm{Poi}\left(\sum_{i=n+1}^\infty \beta_i\right).
\end{equation}
The distributional equality holds because for $\lambda\in\R$, $\espalt{\prob}{e^{\lambda Y_n}} = e^{(e^\lambda-1)\sum_{i=n+1}^\infty \beta_i}$. Set $r_n = -\log\left(\sum_{i=n+1}^\infty \beta_i\right)$ and note that $r_n\rightarrow\infty$. A straightforward calculation shows
\begin{equation}\label{eq: poisson_r_n_lim_cgf}
\lim_{n\uparrow\infty} \frac{1}{r_n}\log\left(\espalt{\prob}{e^{\lambda r_n Y_n}}\right) = \begin{cases} \infty & \lambda > 1 \\ 0 & \lambda \leq 1 \end{cases}.
\end{equation}
In this instance, one cannot use the G\"{a}rtner-Ellis theorem to assert the existence of an LDP for the $\cbra{Y_n}_{n\in\mathbb{N}}$.  However, as the explicit distribution for $Y_n$ is known, a LDP for $\cbra{Y_n}_{n\in\mathbb{N}}$ still holds.

\begin{prop}\label{prop: Y_n_poisson_ldp}

Let $\cbra{\beta_i}_{i\in\mathbb{N}}$ be $\prob$-independent such that $B_{i} \stackrel{\prob}{\sim} \textrm{Poi}(\beta_i)$ for each $i$, and assume  $\sum_{i=1}^\infty \beta_i < \infty$.  Set $r_n = -\log\left(\sum_{i=n+1}^\infty \beta_i\right)$. Then $\cbra{Y_n}_{n\in\mathbb{N}}$ from \eqref{eq: Y_n_poisson_dist} satisfies a LDP with rate $\cbra{r_n}_{n\in\mathbb{N}}$ and good rate function
\begin{equation*}
I(y) = \begin{cases} \infty & y\not\in\cbra{0,1,2,3,...}\\ y & y\in\cbra{0,1,2,3,...}\end{cases}.
\end{equation*}

\end{prop}

\begin{proof}[Proof of Proposition \ref{prop: Y_n_poisson_ldp}]

The result follows via a manual calculation as \eqref{eq: Y_n_poisson_dist} shows that $Y_n\stackrel{\prob}{\sim} \textrm{Poi}\left(e^{-r_n}\right)$ and hence for any $y\in \cbra{0,1,2,3,\dots}$ we have
\begin{equation}\label{eq: Y_n_poisson_as}
\frac{1}{r_n}\log\left(\prob\bra{Y_n=y}\right) = -\frac{1}{r_n}e^{-r_n} - y - \frac{1}{r_n}\log(y!).
\end{equation}
Indeed, let $y\in\cbra{0,1,2,3,\dots}$ and assume $A\subset\reals$ is open with $y\in A$. By \eqref{eq: Y_n_poisson_as} we have
\begin{equation*}
\liminf_{n\uparrow\infty} \frac{1}{r_n}\log\left(\prob\bra{Y_n\in A}\right) \geq \liminf_{n\uparrow\infty}\frac{1}{r_n}\log\left(\prob\bra{Y_n = y}\right) = -y = -I(y),
\end{equation*}
and hence the large deviations lower bound follows from \cite[p. 6]{MR1619036}. Next let $A\subset\reals$ be compact.  If $A\cap\cbra{0,1,2,...} = \emptyset$ then $\lim_{n\uparrow\infty} (1/r_n)\log(\prob\bra{Y_n\in A}) = -\infty = -\inf_{y\in A} I(y)$.  Else, denote by $y_1, ..., y_M$ the (finite) set of non-negative integers in $A$. We have
\begin{equation*}
\begin{split}
\limsup_{n\uparrow\infty}\frac{1}{r_n}\log\left(\prob\bra{Y_n\in A}\right) &= \limsup_{n\uparrow\infty} \frac{1}{r_n}\log\left(\sum_{m=1}^M \prob\bra{Y_n = y_m}\right);\\
&= \max_{m=1,...,M}\cbra{\limsup_{n\uparrow\infty}\frac{1}{r_n}\log\left(\prob\bra{Y_n=y_m}\right)};\\
&= \max_{m=1,...,M}\cbra{-y_m} = -\min_{m=1,...,M}\cbra{I(y_m)}= -\inf_{y\in A} I(y),
\end{split}
\end{equation*}
where the second equality follows from \cite[Lemma 1.2.15]{MR1619036}. Thus, $\cbra{Y_n}_{n\in\mathbb{N}}$ solves the weak LDP with rate function $I$ and scaling $\cbra{r_n}_{n\in\mathbb{N}}$.  Now, let $K>0$.  For any $\lambda > 0$ we have
\begin{equation*}
\prob\bra{Y_n \geq K}\leq e^{-\lambda K + \log\left(\espalt{\prob}{e^{\lambda Y_n}}\right)} = e^{-\lambda K + e^{-r_n}(e^{\lambda} -1)}.
\end{equation*}
Minimizing the right hand side over $\lambda > 0$ we see that the optimal $\hat{\lambda}$ satisfies $\hat{\lambda} = r_n + \log(K)$.  Plugging this value and taking limits gives
\begin{equation*}
\limsup_{n\uparrow\infty} \frac{1}{r_n}\log\left(\prob\bra{Y_n\geq K}\right) \leq \limsup_{n\uparrow\infty}\frac{1}{r_n}\left(-(r_n+\log(K))K + e^{-r_n}(e^{r_n+\log(K)}-1)\right) = -K.
\end{equation*}
As $Y_n\geq 0$ the above inequality implies that $\cbra{Y_n}_{n\in\mathbb{N}}$ is exponentially tight with scaling $\cbra{r_n}_{n\in\mathbb{N}}$ and hence the full LDP follows.
\end{proof}

\begin{rem} Note that $I$ satisfies the hypotheses in Assumption \ref{ass: ldp}.  Also, note that for $\lambda\in\R$
\begin{equation*}
\frac{1}{r_n}\log\left(\espalt{\prob}{e^{\lambda r_n Y_n}}\right) = \frac{1}{r_n}\left(e^{\lambda r_n}-1\right)e^{-r_n} = \frac{e^{(\lambda-1)r_n}}{r_n} - \frac{1}{r_n}e^{-r_n}.
\end{equation*}
From here it follows that (\ref{eq: poisson_r_n_lim_cgf}) holds. Thus, Assumption \ref{ass: ldp} holds for $\cbra{Y_n}_{n\in\mathbb{N}}$.  Additionally, we have $M_* = \underline{M} = -\infty$, and $M^* = \bar{M} = 1$.
\end{rem}

\subsubsection*{The Limiting Indifference Price}
Recall that $d=\lim_{n\uparrow\infty} d^n$ exists. As the assumptions therein hold, using Proposition \ref{prop: ldp_indiff_px}, as well as the explicit formula for $I$, calculation shows for any subsequence (still labeled $n$) such that $\lim_{n\uparrow\infty}|q_n|/r_n$ exists:
\begin{enumerate}[(1)]
\item (Regime 1) If $\lim_{n\uparrow\infty} |q_n|/r_n = 0$ then $\lim_{n\uparrow\infty} p^n(q_n) = d$.
\item (Regime 2) If $\lim_{n\uparrow\infty} |q_n|/r_n = l \neq 0$ then
\begin{equation*}
\lim_{n\uparrow\infty} p^n(q_n)-d = \begin{cases} 0 & l > -\frac{1}{a} \\ \infty & l < -\frac{1}{a}\end{cases}.
\end{equation*}
\item (Regime 3) If $\lim_{n\uparrow\infty} |q_n|/r_n = \infty$ then
\begin{equation*}
\begin{cases}
\limsup_{n\uparrow\infty} p^n(q_n)-d\leq 0 & q_n/r_n\rightarrow \infty\\ \lim_{n\uparrow\infty} p^n(q_n)-d = \infty & q_n/r_n\rightarrow -\infty
\end{cases}.
\end{equation*}
\end{enumerate}

\subsubsection*{Optimal Quantities}

Consider the case where $q_n$ is obtained by purchasing optimal quantities, i.e., $q_n = \hat{q}_n$ from Example \ref{ex: poisson}.  Recall that we require $\tilde{p}^n \geq d$ to ensure $\tilde{p}^n$ is arbitrage free for arbitrary $n$. In this instance, we have
\begin{equation}\label{eq: opt_qn_poi}
\hat{q}_n = -\frac{1}{a}\log\left(\frac{\tilde{p}^n-d^n}{\sum_{i=n+1}^\infty \beta_i}\right) = -\frac{r_n}{a} - \frac{1}{a}\log\left(\tilde{p}^n - d^n\right).
\end{equation}
Thus, if $\tilde{p}^n-d^n  = p > 0$ then $\lim_{n\uparrow\infty} \hat{q}_n/r_n = -1/a$.

\begin{remark} Interestingly, in this instance, for optimal purchases at any price $d^n + p, p > 0$, one encounters the boundary case where $\hat{q}_n/r_n\rightarrow -1/a$, which is not covered by the results in Proposition \ref{prop: ldp_indiff_px}. However, one may explicitly calculate $p^n(\hat{q}_n)$ using \eqref{eq: opt_qn_poi}, \eqref{eq: indiff_px_n} and Example \ref{ex: poisson}
\begin{equation*}
p^n(\hat{q}_n) - d^n=  - \frac{1}{\hat{q}_n a}\log\left(\espalt{\prob}{e^{-\hat{q}_na Y_n}}\right)=\frac{p+d-d^n-e^{-r_n}}{r_n + \log(p+d-d^n)}\rightarrow 0.
\end{equation*}
Additionally, one can directly show that if $q_n/r_n\rightarrow\infty$ then $p^n(q_n)\rightarrow d$.
\end{remark}

\appendix

\section{Proofs From Section \ref{S:SC}}\label{S:SC_Proof}

Before proving Propositions \ref{prop: no_claim} \ref{prop: opt_invest_no_n}, \ref{prop: opt_qn_no_n} and Lemma \ref{lem: arb_free_range_no_n}, we first state and prove some auxiliary lemmas and introduce some notation to streamline the presentation. Throughout this section, Assumptions \ref{ass: filt}, \ref{ass: asset_mkt}, \ref{ass: claim_decomp} and \ref{ass: abstract_int} are enforced.

Recall the measure $\qprob_0$ on $\G_T$ from Assumption \ref{ass: asset_mkt}. Extend $\qprob_0$ to $\F_T$ by defining
\begin{equation}\label{eq: qprob_0_extend}
\qprob_0\bra{A} = \espalt{}{\frac{d\qprob_0}{d\prob}\bigg|_{\G_T} 1_A};\qquad A\in \F_T.
\end{equation}
This extension is similar to the one in \cite[Definition 2.5]{MR1954386}, which therein was called the ``martingale preserving probability measure''.  Next, set
\begin{equation}\label{eq: qprob_0_density}
Z^0_t = \frac{d\qprob_0}{d\prob}\bigg|_{\G_t} = \frac{d\qprob_0}{d\prob}\bigg|_{\F_t};\qquad t\leq T.
\end{equation}
Lastly, recall that $\mathcal{M}$ denotes the class of equivalent local martingale measures on $\F_T$, and $\tm$ denotes the subset of $\mathcal{M}$ with finite relative entropy with respect to $\prob$. For any $\qprob\in\mathcal{M}$ define
\begin{equation}\label{qe: qprob_density}
Z^{\qprob}_t = \frac{d\qprob}{d\prob}\bigg|_{\F_t};\qquad t\leq T,
\end{equation}
so that with an abuse of notation, we have $Z^0_t = Z^{\qprob_0}_t$.


\begin{lemma}\label{lem: mart_meas_structure} If $\qprob\in\mathcal{M}$ then defining $R$ via
\begin{equation}\label{eq: q_r_def}
Z^{\qprob}_t = Z^{0}_t R_t;\qquad t\leq T,
\end{equation}
it follows that $\condespalt{}{R_t}{\G_t} = 1$ for all $t\leq T$. In particular, $\qprob=\qprob_0$ on $\G_T$.
\end{lemma}

\begin{proof}
This fact was proved in essentially the same setting in \cite[Lemma 4.3]{MR2011941}. For the reader's convenience we prove this as well. We follow very closely the proof  of \cite[Lemma A.2]{Christensen11072014} which considered the case of Brownian filtration. Let $A\in\G_t$. By the completeness of the $(\prob,\filtg;S)$-market, for some unique value $x$ there exists a $\filtg$-predictable, $(\prob,\filtg;S)$-integrable strategy $\Delta$ such that $1_A = x + \int_0^T \Delta_tdS_t = X^\Delta_T$. Furthermore, $X^{\Delta}$ is a bounded $(\qprob_0,\filtg)$-martingale, where the boundedness follows as $|X^\Delta_T| \leq 1$.  Now, as $\qprob\in\mathcal{M}$ and $\Delta$ is $(\prob,\filt;S)$-integrable and $X^{\Delta}$ is bounded, it holds that $X^{\Delta}$ is a $(\qprob,\filt)$-local-martingale and hence martingale \cite[Corollary 7.3.8]{MR2200584}.  Thus, we have that
\begin{equation*}
\qprob\bra{A} = \espalt{\qprob}{X^{\Delta}_T} = x = \qprob_0\bra{A};\qquad A\in \G_t,
\end{equation*}
and hence $d\qprob/d\prob\big|_{\G_t} = Z^0_t$, which yields the result as $d\qprob/d\prob\big|_{\G_t} = Z^0_t\condespalt{}{R_t}{\G_t}$.

\end{proof}


\begin{lemma}\label{lem: mart_meas_structure_2}
Let $R_T$ be $\Hh_T$ measurable, strictly positive, and such that $\espalt{}{R_T} = 1$. Then for $\qprob$ defined on $\F_{T}$ by
\begin{equation*}
\frac{d\qprob}{d\prob} = Z^0_T R_T
\end{equation*}
it holds that $\qprob\in\mathcal{M}$.
\end{lemma}

\begin{proof}

Take a sequence of $\filtg$ stopping times $\cbra{\tau_m}_{m\in\mathbb{N}}$ such that $S^m_\cdot = S_{\tau_m\wedge \cdot}$ is a bounded $(\qprob_0,\filtg)$-martingale. For $u\leq T$ define the $(\prob,\filth)$-martingale $R_u = \condespalt{}{R_T}{\Hh_u}$.  It is clear that $d\qprob/d\prob\big|_{\F_u} = Z^0_u R_u$. Now, fix $0\leq s\leq t\leq T$ and let $A_s\in \G_s$, $B_s\in \Hh_s$. We thus have that
\begin{equation*}
\begin{split}
\espalt{}{1_{A_s} 1_{B_s} S_{\tau_m\wedge t}\frac{d\qprob}{d\prob}\big|_{\F_t}} &= \espalt{}{1_{A_s} 1_{B_s} S_{\tau_m\wedge t}Z^0_t R_t}= \espalt{}{1_{A_s}S_{\tau_m\wedge t}Z^0_t}\espalt{}{1_{B_s} R_t};\\
&= \espalt{}{1_{A_s}S_{\tau_m\wedge s}Z^0_s}\espalt{}{1_{B_s} R_s}=\espalt{}{1_{A_s} 1_{B_s} S_{\tau_m\wedge s}\frac{d\qprob}{d\prob}\big|_{\F_s}}.
\end{split}
\end{equation*}
Thus, $S^m$ is a bounded $(\qprob,\filt)$-martingale, proving the result, as the $\cbra{\tau_m}_{m\in\mathbb{N}}$ are $\filt$ stopping times as well.
\end{proof}


Given Lemmas \ref{lem: mart_meas_structure} and \ref{lem: mart_meas_structure_2}, we now prove Proposition \ref{prop: no_claim}

\begin{proof}[Proof of Proposition \ref{prop: no_claim}]

First, consider the optimal investment problem in \eqref{eq: opt_invest_q} in the $(\prob,\filtg; S)$-market: i.e. when the allowable trading strategies are those $\Delta$ which are $\filtg$-predictable, $(\prob,\filtg;S)$-integrable and such that the resultant wealth process $X^{\Delta}$ is a $(\qprob_0,\filtg)$-super-martingale (recall that $\qprob_0$ is the unique equivalent local martingale measure  on $\G_T$ with finite relative entropy). Here, under Assumption \ref{ass: asset_mkt}, as $S$ is $\filtg$ locally bounded, it follows from \cite[Corollary 2.1]{MR1743972}, \cite[Proposition 3.2]{MR1920099} that \eqref{eq: qprob_0_ident} holds for some $(\prob;\filtg;S)$-integrable trading strategy $\Psi$ such that $X^{\Psi}$ is a $(\qprob_0,\filtg)$-martingale. Therefore,
\begin{equation}\label{eq: rel_ent_fact_0}
\infty > \relent{\qprob_0}{\prob\big|_{\G_T}} = \espalt{\qprob_0}{-aX^{\Psi}_T - \log\left(\espalt{}{e^{-aX^\Psi_T}}\right)} = -\log\left(\espalt{}{e^{-aX^{\Psi}_T}}\right).
\end{equation}
This in turn implies
\begin{equation}\label{eq: dual_eqn}
\espalt{}{U\left(X^{\Psi}_T\right)} = -\frac{1}{a}\espalt{}{e^{-aX^{\Psi}_T}} = -\frac{1}{a}e^{-\relent{\qprob_0}{\prob\big|_{\G_T}}},
\end{equation}
and hence from the well-known duality results on the optimal investment problem, $\Psi$ is the optimal trading strategy in the $(\prob,\filtg;S)$-market. We now show $\Psi$ is optimal among the larger class of trading strategies $\mathcal{A}$ in the $(\prob,\filt;S)$-market. Recall the extension of $\qprob_0$ to $\F_T$ in \eqref{eq: qprob_0_extend}. Assumptions \ref{ass: filt}, \ref{ass: asset_mkt} and Lemma \ref{lem: mart_meas_structure_2} imply $\qprob_0\in\mathcal{M}$, and for any $\qprob\in\mathcal{M}$, using $R$ from Lemma \ref{lem: mart_meas_structure}:
\begin{equation*}
\begin{split}
\espalt{}{\frac{d\qprob}{d\prob}\log\left(\frac{d\qprob}{d\prob}\right)} & = \espalt{}{Z^0_TR_T\left(\log\left(Z^0_T\right) + \log(R_T)\right)};\\
&=\espalt{}{Z^0_T\log\left(Z^0_T\right)} + \espalt{}{Z^0_T\condespalt{}{R_T\log(R_T)}{\G_T}}\geq \espalt{}{\frac{d\qprob_0}{d\prob}\log\left(\frac{d\qprob_0}{d\prob}\right)},
\end{split}
\end{equation*}
where the second equality and third inequality follow from Lemma \ref{lem: mart_meas_structure} and the conditional Jensen inequality. Thus, $\qprob_0$ is the $(\prob,\filt)$-minimal entropy measure, and as \eqref{eq: dual_eqn} holds, $\Psi$ will be the optimal trading strategy once it is shown that $\Psi\in\mathcal{A}$: i.e. that $X^{\Psi}$ is a $(\qprob,\filt)$-super-martingale for all $\qprob\in\tm$.  To this end, we first show that $X^{\Psi}$ is a $(\qprob, \filt)$-local martingale for any $\qprob\in\mathcal{M}$.  Indeed, as $X^{\Psi}$ is a $(\qprob_0,\filtg)$-martingale, it is a $(\qprob_0,\filtg)$-special semi-martingale and \cite[Proposition 4.23]{MR959133} implies (recall $x=X^{\Psi}_0 = 0$) that $Y_t = \sup_{s\leq t}|X^{\Psi}_s|, t\leq T$ is $(\qprob_0,\filtg)$ locally integrable.  Thus, let $\cbra{\tau_n}_{n\in\mathbb{N}}$ such that $\tau_n\uparrow\infty$ and such that $\espalt{\qprob_0}{\sup_{s\leq T\wedge\tau_n}|X^{\Psi}_s|}<\infty$.  Now, let $\qprob\in\mathcal{M}$. By Lemma \ref{lem: mart_meas_structure} and the fact that the $\cbra{\tau_n}_{n\in\mathbb{N}}$ are $\filtg$ stopping times, we have conditioning upon $\G_T$ that
\begin{equation*}
\begin{split}
\espalt{\qprob}{\sup_{s\leq T\wedge\tau_n}|X^{\Psi}_s|} &= \espalt{}{Z^0_TR_T\sup_{s\leq T\wedge\tau_n}|X^{\Psi}_s|} =\espalt{\qprob_0}{\sup_{s\leq T\wedge\tau_n}|X^{\Psi}_s|}<\infty.
\end{split}
\end{equation*}
Thus, $(X^{\Psi})^{-}$ is $(\qprob,\filtg)$ (resp. $(\qprob,\filt)$)-locally integrable, and as $S$ is a $(\qprob,\filt)$-local martingale by assumption, \cite[Corollary 7.3.8]{MR2200584} yields that $X^{\Psi}$ is a $(\qprob,\filt)$ local martingale.  To show that $X^{\Psi}$ is a $(\qprob,\filt)$-super-martingale for all $\qprob\in\tm$ we use the results of \cite{MR1891731}. To align with the notation therein, set
\begin{equation}\label{eq: D_underlineZ_def}
\begin{split}
\mathcal{D} &= \cbra{Z^{\qprob}: \qprob\in\tm};\qquad \mathcal{T}_T = \cbra{\tau : \filt-\textrm{stopping time s.t.}\ \tau\leq T};\\
\bar{Z}_t & = \exp\left(\condespalt{\qprob_0}{\log\left(Z^0_T\right)}{\F_t}\right);\quad t\leq T.
\end{split}
\end{equation}
We first claim that $X^{\Psi}$ is a $(\qprob_0,\filt)$-martingale.  Indeed, fix $0\leq s\leq t\leq T$ and let $A_s\in\G_s, B_s\in \Hh_s$.  We have
\begin{equation*}
\begin{split}
\espalt{}{1_{A_s}1_{B_s}X^{\Psi}_t Z^0_t}&= \prob\bra{B_s}\espalt{}{1_{A_s}X^{\Psi}_tZ^0_t}=\prob\bra{B_s}\espalt{}{1_{A_s}X^{\Psi}_sZ^0_s} =\espalt{}{1_{B_s}1_{A_s}X^{\Psi}_sZ^0_s},
\end{split}
\end{equation*}
where the first and third equalities follow by the $\prob$-independence of $\filtg$ and $\filth$, and the second equality follows by the fact that $X^{\Psi}$ is a $(\qprob_0,\filtg)$-martingale. From \eqref{eq: D_underlineZ_def} and \eqref{eq: qprob_0_ident} we see that
\begin{equation*}
\begin{split}
\log\left(\bar{Z}_t\right) &= \condespalt{\qprob_0}{\log\left(Z^{0}_T\right)}{\F_t} = \condespalt{\qprob_0}{-aX^{\psi}_T - \log\left(\espalt{}{e^{-aX^{\psi}_T}}\right)}{\F_t};\\
&=-a X^{\psi}_t - \log\left(\espalt{}{e^{-aX^{\psi}_T}}\right);\qquad t\leq T.
\end{split}
\end{equation*}
Recall (\cite[Section 4]{MR1891731}) that we say $\mathcal{D}$ is ``stable under concatenation'' for $\filt$ if for all $\tau\in\mathcal{T}_T$, we have that $Z^{\qprob_1}, Z^{\qprob_2}\in\mathcal{D}$ implies that
\begin{equation*}
\tilde{Z} = Z^{\qprob_1}I_{[0,\tau)} + (Z^{\qprob_1}_\tau/Z^{\qprob_2}_\tau)Z^{\qprob_2}1_{[\tau,T]} \in \mathcal{D}.
\end{equation*}
It is clear that $\espalt{}{\tilde{Z}_T\log\left(\tilde{Z}_T\right)} < \infty$, $\tilde{Z}_T > 0$ and the optional sampling theorem implies that $\espalt{}{\tilde{Z}_T} = 1$. Lastly, the optional sampling theorem again implies, as $S$ is locally bounded (as noted in \cite[pp. 109]{MR1891730}), that $\tilde{Z}\in\tm$. Thus, as $\tilde{\mathcal{M}}$ is stable under concatenation, \cite[Lemma 4.2]{MR1891731} shows if $\qprob\in\tm$ then $\left(\log(\bar{Z}_{\tau\wedge T})\right)_{\tau\in\mathcal{T}_T}$ is $\qprob$ uniformly integrable.  Thus, the family $\left(X^{\Psi}_{\tau\wedge T}\right)_{\tau\in\mathcal{T}_T}$ is $\qprob$ uniformly integrable and hence $X^{\Psi}$ is a $\qprob$ uniformly integrable, $(\qprob,\filt)$ martingale, hence supermartingale, as it is of class DL.
\end{proof}


\begin{lemma}\label{lem: opt_q_ok} For $\hat{\qprob}$ defined as in \eqref{eq: opt_meas_no_n} it follows that $\hat{\qprob}\in\tm$.
\end{lemma}

\begin{proof}

That $\hat{\qprob}\in\mathcal{M}$ is an immediate consequence of Lemma \ref{lem: mart_meas_structure_2}. It thus suffices to show that $\relent{\hat{\qprob}}{\prob} < \infty$. To this end, using the independence of $Z^0$ and $Y$:
\begin{equation*}
\begin{split}
\relent{\hat{\qprob}}{\prob} &= \espalt{}{Z^0_T\frac{e^{-qaY}}{\espalt{}{e^{-qaY}}}\left(\log(Z^0_T) - qa Y - \log\left(\espalt{}{e^{-qaY}}\right)\right)};\\
&=\relent{\qprob_0}{\prob\big|_{\G_T}} - qa\frac{\espalt{}{Y e^{-qaY}}}{\espalt{}{e^{-qaY}}} - \log\left(\espalt{}{e^{-qa Y}}\right)<\infty,
\end{split}
\end{equation*}
where the last inequality follows as $\espalt{}{e^{\lambda Y}}<\infty$ for all $\lambda\in\reals$.
\end{proof}


\begin{lemma}\label{lem: opt_delta_ok}
For the trading strategy $\hat{\Delta} = -q\Delta_1 + \Psi$ where $\Delta_1$ is the replicating strategy for $D$ and $\Psi$ is from Proposition \ref{prop: no_claim}, it follows that $\hat{\Delta}\in\mathcal{A}$.
\end{lemma}

\begin{proof}

Let $\qprob\in\tm$.  Recall that Assumptions \ref{ass: filt}, \ref{ass: asset_mkt} imply that $\hat{\Delta}$ is $\filtg$- (hence $\filt$-) predictable and  both $(\prob,\filtg;S)$, $(\prob,\filt;S)$ integrable, and $X^{\Delta}$ coincides under both $\filtg,\filt$ (in fact this holds for any measure equivalent to $\prob$ on $\F_T$). We must show that
\begin{equation}\label{eq: hat_delta_x_0}
X^{\hat{\Delta}}_\cdot = \int_0^\cdot \hat{\Delta}_udS_u = -q\int_0^\cdot (\Delta_1)_udS_u + \int_0^\cdot\Psi_udS_u = -q\left(X^{\Delta_1}_\cdot -d\right) + X^{\Psi}_\cdot,
\end{equation}
is a $(\qprob,\filt)$-super-martingale.  From Proposition \ref{prop: no_claim} it holds that $X^{\Psi}$ is a $\qprob$ uniformly integrable $(\qprob,\filt)$-martingale.  Thus, it suffices to show that $X^{\Delta_1}$ is a $(\qprob,\filt)$-martingale. Now, that $X^{\Delta_1}$ is a $(\qprob_0,\filtg)$-martingale follows by Assumptions \ref{ass: asset_mkt} and \ref{ass: abstract_int}. Next, as $X^{\Delta_1}$ is $\filtg$-adapted and $\qprob=\qprob_0$ on $\G_T$ we have, using the cadlag property of $X^{\Delta_1}$ (see Remark \ref{rem: cadlag}), H\"{o}lder's inequality and Doob's maximal inequality:
\begin{equation*}
\begin{split}
\espalt{\qprob}{\sup_{t\leq T}|X^{\Delta_1}_t|} &= \espalt{\qprob_0}{\sup_{t\leq T}|X^{\Delta_1}_t|}=\espalt{\qprob_0}{\sup_{t\leq T}|\condespalt{\qprob_0}{D}{\G_t}|};\\
&\leq \espalt{\qprob_0}{\left(\sup_{t\leq T}|\condespalt{\qprob_0}{D}{\G_t}|\right)^{1+\eps}}^{\frac{1}{1+\eps}};\\
&\leq\left(\frac{1+\epsilon}{\epsilon}\right)\espalt{\qprob_0}{|D|^{1+\eps}}^{\frac{1}{1+\eps}} < \infty.
\end{split}
\end{equation*}
Thus, \cite[Corollary 7.3.8]{MR2200584} implies that $X^{\Delta_1}$ is a $(\qprob,\filt)$-local martingale. In fact, for any $\filt$-stopping time $\tau$ and $\lambda > 0$:
\begin{equation*}
\espalt{\qprob}{|X^{\Delta_1}_{t\wedge \tau}| 1_{|X^{\Delta_1}_{t\wedge\tau}|\geq \lambda}} \leq \espalt{\qprob}{\sup_{t\leq T}|X^{\Delta_1}_t|1_{\sup_{t\geq T}|X^{\Delta_1}_t|\geq \lambda}},
\end{equation*}
and hence  $X^{\Delta_1}$ is of class $(\qprob,\filt)$ D.L. and hence a $(\qprob,\filt)$-martingale.

\end{proof}


\begin{proof}[Proof of Proposition \ref{prop: opt_invest_no_n}]

From \eqref{eq: rel_ent_fact_0} in Proposition \ref{prop: no_claim} we see that
\begin{equation}\label{eq: rel_ent_calc}
\begin{split}
\relent{\qprob_0}{\prob\big|_{\G_T}} & = -\log\left(\espalt{}{e^{-aX^{\Psi}_T}}\right) = -\log(-au(0,0)).
\end{split}
\end{equation}
As $\hat{\Delta}$ is $\filtg$ predictable it follows that $X^{\hat{\Delta}}$ is $\filtg$-adapted and hence $X^{\hat{\Delta}}$ is independent of $\filth$.  Furthermore, Lemma \ref{lem: opt_delta_ok} shows that $\hat{\Delta}\in\mathcal{A}$ and in fact $X^{\hat{\Delta}}$ is a $(\qprob,\filt)$-martingale for all $\qprob\in\tm$. Additionally, in view of \eqref{eq: hat_delta_x_0} we have
\begin{equation*}
\begin{split}
-a(X^{\hat{\Delta}}_T + qB) &= -aX^{\psi}_T + qa X^{\Delta_1}_T - qad - qaD - qa Y = -aX^{\psi}_T - qad - qa Y.
\end{split}
\end{equation*}
Thus
\begin{equation}
\begin{split}
-\frac{1}{a}\espalt{}{e^{-a(X^{\hat{\Delta}}_T + qB)}} &= -\frac{1}{a}e^{-qad} \espalt{}{e^{-a X^{\psi}_T - qaY}}=-\frac{1}{a}e^{-qad}\espalt{}{e^{-aX^{\psi}_T}}\espalt{}{e^{-qaY}};\\
&=u(0,0) e^{-qa d}\espalt{}{e^{-qa Y}}.\label{Eq:OptimalStrategy1}
\end{split}
\end{equation}
Now, define the probability measure $\hat{\qprob}$ on $\F_T$ via \eqref{eq: opt_meas_no_n}. Lemma \ref{lem: opt_q_ok} shows that $\qprob\in\tm$.  From \eqref{eq: opt_meas_no_n} we have
\begin{equation*}
\begin{split}
B+\frac{1}{qa}\log\left(Z^{\hat{\qprob}}_T\right) &= D + Y + \frac{1}{qa}\log\left(Z^0_T\right) - Y - \frac{1}{qa}\log\left(\espalt{}{e^{qa Y}}\right);\\
&= X^{\Delta_1}_T + \frac{1}{qa}\log\left(Z^0_T\right) - \frac{1}{qa}\log\left(\espalt{}{e^{qa Y}}\right).\\
\end{split}
\end{equation*}
As $\espalt{}{Z^{\hat{\qprob}}_T} = 1$:
\begin{equation*}
\begin{split}
\espalt{}{Z^{\hat{\qprob}}_T\left(B+\frac{1}{qa}\log\left(Z^{\hat{\qprob}}_T\right)\right)} &= \espalt{}{X^{\Delta_1}_TZ^0_T\frac{e^{-qaY}}{\espalt{}{e^{-qaY}}}} + \frac{1}{qa}\espalt{}{\log(Z^0_T)Z^0_T\frac{e^{-qaY}}{\espalt{}{e^{-qaY}}}}\\
&\qquad\qquad - \frac{1}{qa}\log\left(\espalt{}{e^{qa Y}}\right);\\
&=d + \frac{1}{qa}\relent{\qprob_0}{\prob\big|_{\G_T}} - \frac{1}{qa}\log\left(\espalt{}{e^{qa Y}}\right);\\
&=d - \frac{1}{qa}\log(-au(0,0)) - \frac{1}{qa}\log\left(\espalt{}{e^{qa Y}}\right).\\
\end{split}
\end{equation*}
Above, the second equality follows by the independence of $Y$ and $Z^0_TX^{\Delta_1}_T$, the fact that $X^{\Delta_{1}}$ is a $(\qprob_0,\filtg)$ martingale starting at $d$, and the definition of $Z^0_T$. The last equality follows from \eqref{eq: rel_ent_calc}.  The latter and (\ref{Eq:OptimalStrategy1}), give us
\begin{equation*}
\begin{split}
-\frac{1}{a}e^{-qa\espalt{}{Z^{\hat{\qprob}}_T\left(B+\frac{1}{qa}\log\left(Z^{\hat{\qprob}}_T\right)\right)}} &= -\frac{1}{a}e^{-qad + \log(-au(0,0)) + \log\left(\espalt{}{e^{qaY}}\right)}= u(0,0)e^{-qad}\espalt{}{e^{-qaY}};\\
&=-\frac{1}{a}\espalt{}{e^{-a(X^{\hat{\Delta}}_T + qB)}}.
\end{split}
\end{equation*}
Thus, from the standard duality results for exponential utility it follows that \eqref{eq: u_no_n_vf} holds, that $\hat{\Delta}$ is the optimal strategy, and that $\hat{\qprob}$ is the optimal local martingale measure.  With this identification of $u(0,q)$, the indifference price $p(q)$ from \eqref{eq: indif_price_no_n} is immediate.
\end{proof}


\begin{proof}[Proof of Lemma \ref{lem: arb_free_range_no_n}]

Let $\qprob\in\mathcal{M}$. From Lemma \ref{lem: mart_meas_structure} it follows that $Z^{\qprob}_T = Z^0_TR_T$ where $\condespalt{}{R_T}{\G_T} = 1$. As $D = X^{\Delta_1}_T$ almost surely and  $X^{\Delta_1}_T$ is a $(\qprob_0,\filtg)$-martingale with initial value $d$ it follows that
\begin{equation*}
\espalt{\qprob}{B} = \espalt{}{Z^0_TR_T\left(X^{\Delta_1}_T+Y\right)} = d + \espalt{}{Z^0_TR_T Y},
\end{equation*}
where the second equality follows by first conditioning upon $\G_T$.  From the above, it is clear that $\inf_{\qprob\in\mathcal{M}}\espalt{\qprob}{B} \geq d +\essinf{\prob}{Y}$.  As for the reverse direction, denote by $\textbf{M}_T$ the class of strictly positive, $\Hh_T$-measurable random variables $R_T$ such that $\espalt{}{R_T} = 1$. For any $R_T\in \textbf{M}_T$, Lemma \ref{lem: mart_meas_structure_2} shows that defining $\qprob$ via $d\qprob/d\prob = Z^0_TR_T$ it follows that $\qprob\in\mathcal{M}$. Furthermore, using the independence of $\filtg$ and $\filth$ it follows that $\espalt{\qprob}{B} = d + \espalt{}{R_TY}$, so that
\begin{equation}\label{eq: M_bfM}
\inf_{\qprob\in\mathcal{M}}\espalt{\qprob}{B} \leq d + \inf_{R_T\in\textbf{M}_T}\espalt{}{R_T Y}.
\end{equation}
Now, let $m$ be such that $\prob\bra{Y < m} > 0$. Set $A_{m} = \cbra{Y < m}\in \Hh_T$ and, for $0<\delta<1$ set
\begin{equation*}
R^{m,\delta}_T = \frac{(1-\delta)1_{A_{m}} + \delta 1_{A_{m}^c}}{(1-\delta)\prob\bra{A_{m}} + \delta\prob\bra{A_{m}^c}}.
\end{equation*}
Clearly, $R^{m,\delta}_T\in\textbf{M}_T$. Furthermore,
\begin{equation*}
\begin{split}
\inf_{R_T\in\textbf{M}_T}\espalt{}{R_TY}\leq\espalt{}{R^{m,\delta}_T Y} &= \frac{(1-\delta)\espalt{}{Y1_{Y < m}} + \delta\espalt{}{Y1_{Y\geq m}}}{(1-\delta)\prob\bra{Y<m} + \delta\prob\bra{Y\geq m}}\leq \frac{m(1-\delta)\prob\bra{Y < m} + \delta\espalt{}{Y1_{Y\geq m}}}{(1-\delta)\prob\bra{Y<m} + \delta\prob\bra{Y\geq m}}.
\end{split}
\end{equation*}
Assumption \ref{ass: abstract_int} implies $\espalt{}{|Y|} < \infty$ and in particular $\espalt{}{|R^{m,\delta}_T Y|} < \infty$. Thus, taking $\delta \downarrow 0$ gives
\begin{equation*}
\inf_{R_T\in\textbf{M}_T}\espalt{}{R_TY} \leq m.
\end{equation*}
Taking $m\downarrow\essinf{\prob}{Y}$ gives $\inf_{R_T\in\textbf{M}_T}\espalt{}{R_TY} \leq \essinf{\prob}{Y}$, which in view of \eqref{eq: M_bfM} yields
\begin{equation*}
\underline{b} = \inf_{\qprob\in\mathcal{M}}\espalt{\qprob}{B} = d + \essinf{\prob}{Y}.
\end{equation*}
A similar calculation for the upper bound shows that $\bar{b} = d + \esssup{\prob}{Y}$, finishing the proof.

\end{proof}


\begin{proof}[Proof of Proposition \ref{prop: opt_qn_no_n}]
It is convenient to set $q = -\lambda/a$ so that \eqref{eq: opt_q_var_prob_no_n} reads
\begin{equation*}
\frac{1}{a}\inf_{\lambda\in\R}\left(\Lambda(\lambda) - \lambda(\tilde{p} - d)\right).
\end{equation*}
Set $f(\lambda) =\Lambda(\lambda) - \lambda(\tilde{p}-d)$. The strict convexity of $\Lambda(\lambda)$ implies that $f(\lambda)$ is strictly convex.  For $\lambda\neq 0$ we have $f(\lambda)/\lambda = \Lambda(\lambda)/\lambda - (\tilde{p}-d)$. Lemma \ref{lem: arb_free_range_no_n} and parts $(1),(2)$ of Lemma \ref{lem: Xi_prop_no_n} below give, as $\tilde{p}\in (\underline{b},\bar{b})$, the existence of an $\eps > 0$ so that $\liminf_{|\lambda|\uparrow\infty}f(\lambda)/|\lambda| \geq \eps$. Therefore, $f$ is strictly convex and coercive so there exists a unique minimizer $\hat{\lambda}$ for $f$ on $\R$.  Part $(3)$ of Lemma \ref{lem: Xi_prop_no_n} below ensures that $\dot{\Lambda}(\lambda)$ exists and is finite for all $\lambda\in\R$, and hence by the standard results of minimizer's of differentiable functions it follows that $\hat{\lambda}$ must satisfy the first order conditions given in \eqref{eq: hat_qn_foc_no_n}. To see this, note that for all $\lambda\in\R$ we have $\Lambda(\lambda)-\Lambda(\hat{\lambda}) \geq (\lambda-\hat{\lambda})(\tilde{p}-d)$. Now, assume $\lambda > \hat{\lambda}$. Then $\tilde{p}-d \leq 1/(\lambda-\hat{\lambda})\int_{\hat{\lambda}}^\lambda \dot{\Lambda}(\tau)d\tau.$ Taking $\lambda\downarrow\hat{\lambda}$ and using the smoothness of $\Lambda$, which is ensured by Assumption \ref{ass: abstract_int} $\tilde{p}-d \leq \dot{\Lambda}(\hat{\lambda}).$ A similar calculation with $\hat{\lambda} > \lambda$ gives the opposite inequality, finishing the proof.

\end{proof}


\section{Proofs from Section \ref{S:LM_Ex}}\label{S: LM_Ex_Proof}

\begin{proof}[Proof of Lemma \ref{lem: h_well_posed}]

It is first shown that $\lim_{N\uparrow\infty} \sum_{i=1}^N \espalt{}{B_i}$ exists and is finite in magnitude.  Indeed, by the convexity of each $\Gamma_i$ we have for any $\lambda > 0$ that $-(1/\lambda)\Gamma_i(-\lambda)\leq \dot{\Gamma}_i(0) = \espalt{}{B_i} \leq (1/\lambda)\Gamma_i(\lambda)$. This gives for any integers $M>N$ that
\begin{equation*}
\frac{1}{-\lambda}\sum_{i=N+1}^M \Gamma_i(-\lambda) \leq \sum_{i=N+1}^M \espalt{}{B_i} \leq \frac{1}{\lambda}\sum_{i=N+1}^M \Gamma_i(\lambda).
\end{equation*}
As both $(1/\lambda)\sum_{i=1}^\infty \Gamma_i(\lambda)$ and $-(1/\lambda)\sum_{i=1}^\infty \Gamma_i(-\lambda)$ exist and are finite for all $\lambda > 0$, it follows for any $\epsilon > 0$ there is some $N_{\epsilon}$ so that if $M,N\geq N_{\epsilon}$ then $\left|\sum_{i=N+1}^M \espalt{}{B_i}\right|\leq \epsilon$, proving that $\sum_{i=1}^n \espalt{}{B_i}$ is Cauchy and hence the limit exists and is finite.  We now claim that
\begin{equation}\label{eq: h_var}
\sum_{i=1}^\infty \varalt{}{B_i} < \infty,
\end{equation}
from which the almost sure convergence result follows from \cite[Theorem 1.4.2]{MR2760872} and the $L^2$ convergence result follows as $\sum_{i=1}^{\infty}\espalt{}{B_i}$ exists. But, \eqref{eq: h_var} holds by applying the inequality
\begin{equation*}
x^2 \leq \frac{2}{\lambda^2}\left(e^{\lambda x} + e^{-\lambda x}\right);\qquad x \in\reals,\lambda > 0,
\end{equation*}
to $x = \sum_{i=1}^N \left(B_i - \espalt{}{B_i}\right)$, using the independence of the $\cbra{B_i}_{i\in\N}$, Assumption \ref{ass: h_cgf}, and that $\sum_{i=1}^N \espalt{}{B_i}\rightarrow \sum_{i=1}^\infty \espalt{}{B_i}$ as $N\uparrow\infty$.
\end{proof}

\begin{proof}[Proof of Lemma \ref{lem: LM_SC}]

Clearly, $\filtg$ and $\filth$ satisfy Assumption \ref{ass: filt} and by construction $B = D + Y$ so that Assumption \ref{ass: claim_decomp} holds.  Due to the choice of $\sigma$ as the lower-triangular square root of $\Sigma$, Assumption \ref{ass: asset_mkt} is satisfied.  Indeed, the first $n$ assets only depend upon the first $n$ Brownian motions and hence $S$ is $\filtg$ adapted.   Furthermore, the $(\prob,\filtg;S)$-market is complete in view of the martingale representation theorem.  Here, the unique martingale measure $\qprob_0$ takes the form
\begin{equation*}
\frac{d\qprob_0}{d\prob}\bigg|_{\G_T} = \frac{d\tilde{\qprob}}{d\prob}\bigg|_{\G_T} = \mathcal{E}\left(\sum_{i=1}^n -\theta_iW^i_\cdot\right)_T,
\end{equation*}
and for this measure $\relent{\qprob_0}{\prob\big|_{\G_T}} = (1/2)\sum_{i=1}^n \theta_i^2 < \infty$. Lastly, Assumptions \ref{ass: h_n_mart} and \ref{ass: h_cgf} imply the integrability assumptions on $D$ and $Y$ in Assumption \ref{ass: abstract_int}.  To see this, as Assumption \ref{ass: h_n_mart} implies $\espalt{}{B_i^2} < \infty$ for each $i$, it holds that for any $0 < \eps < 1$ (recall \eqref{eq: qprob_0_density}):
\begin{equation*}
\begin{split}
\espalt{\qprob_0}{|D|^{1+\eps}} &= \espalt{}{Z^0_T\left|\sum_{i=1}^n B_i\right|^{1+\eps}}\leq \espalt{}{\left(Z^0_T\right)^{\tfrac{2}{1-\eps}}}^{\tfrac{1-\eps}{2}}\espalt{}{\left(\sum_{i=1}^n B_i\right)^2}^{\tfrac{1+\eps}{2}};\\
&\leq e^{\tfrac{T}{2}\frac{1+\eps}{1-\eps}\sum_{i=1}^n\theta^2_i}\left(n\sum_{i=1}^n \espalt{}{B^2_i}\right)^{\frac{1+\eps}{2}} < \infty.
\end{split}
\end{equation*}
Furthermore, the independence of the $\cbra{B_i}_{i\in\N}$ gives
\[
\espalt{}{e^{\lambda Y}} = \espalt{}{e^{\lambda\sum_{i=n+1}^\infty B_i}} = e^{\sum_{i=n+1}^\infty \Gamma_i(\lambda)} < \infty.
 \]
Having verified Assumptions \ref{ass: filt} -- \ref{ass: abstract_int}, Proposition \ref{prop: opt_invest_no_n} implies that for each $n$, in the $n^{th}$ market, the indifference price for $q$ units of $B$ is $p(q) = d  - (1/(qa))\log\left(\espalt{}{e^{-qa Y}}\right)$. By construction, $\qprob_0$ agrees with $\tilde{\qprob}$ on the sigma-algebra $\G_T$. Thus, we have from the definitions of $d$ in Assumption \ref{ass: asset_mkt} and $d^n$ in \eqref{eq: x1_def_lm}:
\begin{equation*}
\begin{split}
d &= \espalt{\qprob_0}{D} = \sum_{i=1}^n \espalt{\tilde{\qprob}}{B_i} = d^n;\qquad \log\left(\espalt{}{e^{-qaY}}\right) = \sum_{i=n+1}^\infty \Gamma_i(-qa).
\end{split}
\end{equation*}
The range of arbitrage free prices in \eqref{eq: arb_free_range_n} follows immediately from Lemma \ref{lem: arb_free_range_no_n}, Lemma \ref{lem: h_well_posed} (which shows that $\sum_{i=1}^\infty \espalt{}{B_i}$ exists) and the independence of the $\cbra{B_i}_{i\in\N}$. Indeed, these imply that $\essinf{\prob}{\sum_{i=n+1}^\infty B_i} = \sum_{i=n+1}^\infty \essinf{\prob}{B_i}$ as well as $\esssup{\prob}{\sum_{i=n+1}^\infty B_i} = \sum_{i=n+1}^\infty \esssup{\prob}{B_i}$.

For the sake of completeness, let us discuss $\essinf{\prob}{\sum_{i=n+1}^\infty B_i} = \sum_{i=n+1}^\infty \essinf{\prob}{B_i}$ in some detail. Without loss of generality we may assume that $\espalt{}{B_i}=0$ and we have that $\sum_{i=1}^\infty \espalt{}{B^{2}_i}<\infty$. Thus, $M_\infty = \sum_{i=n+1}^\infty B_i$ is well defined almost surely. It is clear that $\essinf{\prob}{\sum_{i=n+1}^\infty B_i} \geq \sum_{i=n+1}^\infty \essinf{\prob}{B_i}$ as for any $c<\sum_{i=n+1}^\infty \essinf{\prob}{B_i}$ one has $\prob\left[M_\infty<c\right]=0$. For the other direction, we let $M_m = \sum_{i=n+1}^{m+n} B_i$. Then $M = (M_m)_{m=1,2,\cdots}$ is a $L^{2}$ bounded $\F_m = \sigma(B_i, i=n+1,...,n+m)$ martingale. Hence, we may write $M_m = E[M_\infty | \F_m]$ which
immediately gives that
\[\sum_{i=n+1}^{n+m}\essinf{\prob}{B_i} = \essinf{\prob}{\sum_{i=n+1}^{n+m} B_i }=\essinf{\prob}{ M_{m}} \geq \essinf{\prob}{ M_{\infty}}= \essinf{\prob}{\sum_{i=n+1}^{\infty}B_i},\]
for  $m = 1,2,\cdots$. So, taking $m\uparrow\infty$ gives the result. Lastly, \eqref{eq: hat_qn_foc} follows immediately from \eqref{eq: hat_qn_foc_no_n} as  $\Lambda$ from \eqref{eq: h2_cgf} takes the form $\Lambda(\lambda) = \sum_{i=n+1}^\infty \Gamma_i(\lambda)$ and Lemma \ref{lem: Xi_prop} below shows that the derivative may be passed through the infinite sum.

\end{proof}

\begin{lemma}\label{lem: Xi_prop}
Let Assumptions \ref{ass: mu_sig}, \ref{ass: h_n_mart} and \ref{ass: h_cgf} hold.  For $i\in\N$ define $\Gamma_i$ as in \eqref{eq: nth_cgf}.  Then, for all $\lambda \in \R$, $\lim_{N\uparrow\infty}\sum_{i=1}^N\dot{\Gamma}_i(\lambda)$ exists, $-\infty < \sum_{i=1}^\infty \dot{\Gamma}_i(\lambda) < \infty$ and $(d/d\lambda)\left(\sum_{i=1}^\infty \Gamma_i(\lambda)\right) = \sum_{i=1}^\infty \dot{\Gamma}_i(\lambda)$.
\end{lemma}

\begin{proof}
The short proof of this lemma was suggested by one of the referees of the paper. As $\Gamma_{i}$ is convex it follows that
\[
\sum_{i=k}^{n}\left(\Gamma_{i}(\lambda)-\Gamma_{i}(\lambda-1)\right)\leq \sum_{i=k}^{n}\dot{\Gamma}_{i}(\lambda)\leq \sum_{i=k}^{n}\left(\Gamma_{i}(\lambda+1)-\Gamma_{i}(\lambda)\right),
\]
which by Assumption \ref{ass: h_cgf} means that $\sum_{i=0}^{\cdot}\dot{\Gamma}_{i}(\lambda)$ is a Cauchy series and thus converges to a finite quantity. This completes the proof of the Lemma.
\end{proof}

\section{Proof of Proposition \ref{prop: ldp_indiff_px}}\label{Appendix:prop: ldp_indiff_px}
\begin{proof}[Proof of Proposition \ref{prop: ldp_indiff_px}]
(Regime 1) For the $\delta$ of Assumption \ref{ass: ldp}, let $\eps > 0$ be such that $\eps a < \delta$.  For $n$ large enough we may assume that $|q_n|\leq \eps r_n$.  From \eqref{eq: p_n_dec} it follows that
\begin{equation*}
-\frac{1}{a\eps r_n}\Lambda_n(-a\eps r_n) = p^n(\eps r_n) - d^n \leq p^n(q_n) - d^n \leq p^n(-\eps r_n) - d^n = \frac{1}{a\eps r_n}\Lambda_n(a\eps r_n).
\end{equation*}
Therefore, Varadhan's integral lemma yields
\begin{equation*}
\begin{split}
\liminf_{n\uparrow\infty} p^n(q_n)-d^n &\geq -\frac{1}{a\eps}\sup_{y\in\R}\left(-a\eps y - I(y)\right);\\
\limsup_{n\uparrow\infty} p^n(q_n)-d^n &\leq \frac{1}{a\eps}\sup_{y\in\R}\left(a\eps y - I(y)\right).
\end{split}
\end{equation*}
In view of Lemmas \ref{lem: rf_coerciv}, \ref{lem: ldp_zero_fact} below, we have for $\eps$ small enough that
\begin{equation*}
\begin{split}
\liminf_{n\uparrow\infty}p^n(q_n)-d^n &\geq y_{-\eps} + \frac{1}{\eps a} I(y_{-\eps}) \geq y_{-\eps};\\
\limsup_{n\uparrow\infty}p^n(q_n)-d^n &\leq y_{+\eps} - \frac{1}{\eps a} I(y_{+\eps}) \leq y_{+\eps},
\end{split}
\end{equation*}
for some $y_{-\eps}\in [l^{-\eps a},u^{-\eps a}]$, $y_{+\eps}\in [l^{\eps a},u^{\eps a}]$ where $l^\eps,u^\eps$ are defined in \eqref{eq: l_u_eps} below. Thus, by Lemma \ref{lem: ldp_zero_fact} we have that $y_{\pm\eps}\rightarrow 0$ as $\eps \downarrow 0$ proving that $\lim_{n\uparrow\infty} p^n(q_n)-d^n  = 0$.

(Regime 2). Now, assume that $\lim_{n\uparrow\infty} |q_n|/r_n = l \in (0,\infty)$.  First, assume $0 < l < -\underline{M}/a$. For $n$ large enough we may assume $(l - \gamma)r_n \leq q_n \leq (l + \gamma)r_n$ for some $\gamma >0$ such that $\underline{M} < -a(l +\gamma) < -a(l - \gamma) < 0$.  \eqref{eq: p_n_dec} then implies
\begin{equation*}
\begin{split}
p^n(q_n)-d^n &\leq p^n((l-\gamma)r_n)-d^n = -\frac{1}{(l-\gamma)ar_n}\Lambda_n(-(l-\gamma)ar_n);\\
p^n(q_n)-d^n &\geq p^n((l+\gamma)r_n)-d^n = -\frac{1}{(l+\gamma)ar_n}\Lambda_n((l+\gamma)ar_n).\\
\end{split}
\end{equation*}
By Varadhan's integral lemma
\begin{equation*}
\begin{split}
\limsup_{n\uparrow\infty} p^n(q_n) -d^n &\leq -\frac{1}{(l-\gamma)a}\sup_{y\in\R}\left(-(l-\gamma)ay - I(y)\right) = \inf_{y\in\R}\left(y + \frac{I(y)}{(l-\gamma)a}\right)\leq 0;\\
\liminf_{n\uparrow\infty} p^n(q_n) -d^n &\geq -\frac{1}{(l+\gamma)a}\sup_{y\in\R}\left(-(l+\gamma)ay - I(y)\right) = \inf_{y\in\R}\left(y + \frac{I(y)}{(l+\gamma)a}\right) > -\infty.\\
\end{split}
\end{equation*}
The function $\tau \mapsto \inf_{y\in\R}\left(y + \tau I(y)\right)$ for $\tau > 0$ is concave and hence continuous on the interior of it's effective domain. Therefore, taking $\gamma\downarrow 0$ in the above yields \eqref{eq: limit_px_ell_finite}.

Now, assume that $M_* > -\infty$ and  $l > -M_*/a$. From \eqref{eq: M_M_bounds} we have that $l > -\underline{M}/a$ and hence we can find a $\gamma > 0$ so that for $n$ large enough $q_n \geq (l-\gamma)r_n$ and such that $l-\gamma > -M_*/a \geq -\underline{M}/a$. As before, \eqref{eq: p_n_dec} implies
\begin{equation}\label{eq: large_ell_px_up}
p^n(q_n)-d^n \leq p^n((l-\gamma)r_n) - d^n = -\frac{1}{(l-\gamma)ar_n}\Lambda_n(-(l-\gamma)ar_n).
\end{equation}
By the definition of $\underline{M}$ we know that $\limsup_{n\uparrow\infty}(1/r_n)\Lambda_n(-(l-\gamma)ar_n) = \infty$. However, it is in fact true that $\lim_{n\uparrow\infty} (1/r_n)\Lambda_n(-(l-\gamma)ar_n) = \infty$. Indeed, assume there exists a sub-sequence (still labeled $n$) such that $\limsup_{n\uparrow\infty}(1/r_n)\Lambda_n(-(l-\gamma)ar_n) < \infty$. Varadhan's integral lemma applied to the subsequence (for which the LDP still holds) then implies that for $\gamma$ small enough
\begin{equation*}
\lim_{n\uparrow\infty}\frac{1}{r_n}\Lambda_n(-(l-2\gamma)ar_n) = \sup_{y\in\R}\left(-(l-2\gamma)a y - I(y)\right) < \infty.
\end{equation*}
Thus, for $\gamma$ small enough so that $l - 2\gamma > -M_*/a$ we have a contradiction to the definition of $M_*$. Thus, we have from \eqref{eq: large_ell_px_up} that $\lim_{n\uparrow\infty} p^n(q_n) - d^n = -\infty$.

The results for $q_n/r_n \rightarrow l < 0$ are very similar to that for $l > 0$.  Indeed, assume first that $-\bar{M}/a < l < 0$. For $n$ large enough we may assume $(l - \gamma)r_n \leq q_n \leq (l + \gamma)r_n$ for some $\gamma >0$ such that $0 < -a(l +\gamma) < -a(l - \gamma) < \bar{M}$.  \eqref{eq: p_n_dec} implies
\begin{equation*}
\begin{split}
p^n(q_n)-d^n &\leq p^n((l-\gamma)r_n)-d^n = -\frac{1}{(l-\gamma)ar_n}\Lambda_n(-(l-\gamma)ar_n);\\
p^n(q_n)-d^n &\geq p^n((l+\gamma)r_n)-d^n = -\frac{1}{(l+\gamma)ar_n}\Lambda_n(-(l+\gamma)ar_n).\\
\end{split}
\end{equation*}
Varadhan's integral lemma gives
\begin{equation*}
\begin{split}
\limsup_{n\uparrow\infty} p^n(q_n) -d^n &\leq -\frac{1}{(l-\gamma)a}\sup_{y\in\R}\left(-(l-\gamma)ay - I(y)\right) = \sup_{y\in\R}\left(y + \frac{I(y)}{(l-\gamma)a}\right)<\infty;\\
\liminf_{n\uparrow\infty} p^n(q_n) -d^n &\geq -\frac{1}{(l+\gamma)a}\sup_{y\in\R}\left(-(l+\gamma)ay - I(y)\right) = \sup_{y\in\R}\left(y + \frac{I(y)}{(l+\gamma)a}\right) \geq 0.\\
\end{split}
\end{equation*}
The function $\tau \mapsto \sup_{y\in\R}\left(y + \tau I(y)\right)$ for $\tau < 0$ is convex and hence continuous on it's effective domain. Therefore, taking $\gamma\downarrow 0$ in the above yields \eqref{eq: limit_px_ell_finite0}.

Next, assume that $M^*< \infty$ and $l < -M^*/a$. From \eqref{eq: M_M_bounds} we have that $l < -\bar{M}/a$ and hence we can find a $\gamma > 0$ so that for $n$ large enough $q_n \leq (l+\gamma)r_n$ and such that $l+\gamma < -M^*/a \leq -\bar{M}/a$. As before, \eqref{eq: p_n_dec} implies
\begin{equation*}
p^n(q_n)-d^n \geq p^n((l+\gamma)r_n) - d^n = -\frac{1}{(l+\gamma)r_n}\Lambda_n(-(l+\gamma)ar_n).
\end{equation*}
By the definition of $\bar{M}$  we know that $\limsup_{n\uparrow\infty}(1/r_n)\Lambda_n(-(l+\gamma)ar_n) = \infty$, but a similar argument to that above shows that in fact $\lim_{n\uparrow\infty} (1/r_n)\Lambda_n(-(l+\gamma)ar_n) = \infty$, and hence $\lim_{n\uparrow\infty} p^n(q_n) - d^n = \infty$.

(Regime 3) The proof for Regime 3 is nearly identical to that of Regime 2 and hence only the argument for $q_n/r_n \rightarrow \infty$ is given.  For any $M>0$ we can find $N$ large enough so that $q_n \geq M r_n$. \eqref{eq: p_n_dec} then implies
\begin{equation*}
p^n(q_n)-d^n \leq p^n(Mr_n) -d^n = -\frac{1}{Mr_n a}\Lambda_n(-Mar_n).
\end{equation*}
Now, if $M_*> -\infty$ then \eqref{eq: M_M_bounds} implies $\underline{M} >-\infty$ and hence for $M$ large enough so that $-aM < M_* < \underline{M}$ we have $\limsup_{n\uparrow\infty}(1/r_n)\Lambda_n(-Mar_n) = \infty$, but, in fact we must have $\lim_{n\uparrow\infty}(1/r_n)\Lambda_n(-Mar_n) = \infty$. Indeed, assume there exists a sub-sequence (still labeled $n$) such that $\limsup_{n\uparrow\infty}(1/r_n)\Lambda_n(-Mar_n) <\infty$. Then for $\gamma$ small enough so that $-a(M-\gamma) < M_*$
\begin{equation*}
\lim_{n\uparrow\infty}\frac{1}{r_n}\Lambda_n(-(M-\gamma)ar_n) = \sup_{y\in\R}\left(-(M-\gamma)a y - I(y)\right) < \infty.
\end{equation*}
This is a contradiction to the definition of $M_*$. Therefore, $\lim_{n\uparrow\infty} p^n(q_n) -d^n = -\infty$, proving the result if $M_* > -\infty$.  If $\underline{M} = -\infty$ we have that for all $M>0$ that
\begin{equation*}
\limsup_{n\uparrow\infty} p^n(q_n)-d^n \leq -\frac{1}{Ma}\sup_{y\in\R}\left(-Ma y - I(y)\right) = \inf_{y\in\R}\left(y + \frac{I(y)}{Ma}\right).
\end{equation*}
The right hand side above is decreasing in $M$ : taking $M\uparrow\infty$ gives
\begin{equation*}
\limsup_{n\uparrow\infty}p^n(q_n) - d^n \leq \lim_{M\uparrow\infty}\inf_{y\in\reals}\left(y + \frac{I(y)}{Ma}\right).
\end{equation*}
We now claim that
\begin{equation*}
\lim_{M\uparrow\infty}\inf_{y\in\reals}\left(y + \frac{I(y)}{Ma}\right) = \inf\cbra{y \ | \ I(y) < \infty},
\end{equation*}
which, if true, finishes the result.  First, it is clear as $I\geq 0$ that
\begin{equation*}
\lim_{M\uparrow\infty}\inf_{y\in\reals}\left(y + \frac{I(y)}{Ma}\right) = \lim_{M\uparrow\infty}\inf_{y\in\reals, I(y)<\infty }\left(y + \frac{I(y)}{Ma}\right) \geq \inf\cbra{y \ | \ I(y) < \infty}.
\end{equation*}
Now, let $y$ be such that $I(y) < \infty$. As $\inf_{y\in\reals}\left(y + I(y)/(Ma)\right) \leq y + I(y)/(Ma)$ it follows that $\lim_{M\uparrow\infty}\inf_{y\in\reals}\left(y + I(y)/(Ma)\right) \leq y$.  Taking $y\downarrow \inf\cbra{y\ | \ I(y) < \infty}$ gives the result.
\end{proof}

\section{Supporting Lemmas}

\begin{lemma}\label{lem: Xi_prop_no_n}
Define $\Lambda$ as in \eqref{eq: h2_cgf} and assume that $\Lambda(\lambda)$ is finite for all $\lambda\in \mathbb{R}$.  Then
\begin{enumerate}[(1)]
\item $\lim_{\lambda\uparrow\infty} (1/\lambda) \Lambda(\lambda) = \esssup{\prob}{Y}$.
\item $\lim_{\lambda\downarrow -\infty} (1/\lambda)\Lambda(\lambda) = \essinf{\prob}{Y}$.
\item For all $\lambda \in \R$, $\dot{\Lambda}(\lambda) = \espalt{}{Y e^{\lambda Y}}/\espalt{}{e^{\lambda Y}} \in\reals$. Additionally, $\Lambda(\lambda)$ is strictly convex which implies that the map $\lambda\mapsto \dot{\Lambda}(\lambda)$ is increasing in $\lambda$.
\end{enumerate}
\end{lemma}

\begin{proof}
Clearly, for $\lambda > 0$, we have
\[
(1/\lambda)\Lambda(\lambda) = (1/\lambda)\log\left(\espalt{\prob}{e^{\lambda Y}}\right) \leq \esssup{\prob}{Y}.
\]
 Now, let $m>0$ be such that $\prob\bra{Y>m} > 0$. As $Y\geq m 1_{Y>m}$,  we then have
\begin{equation*}
\frac{1}{\lambda}\Lambda(\lambda) \geq m + \frac{1}{\lambda}\log\left(\espalt{\prob}{1_{Y>m}}\right).
\end{equation*}
Taking $\lambda\uparrow\infty$ gives that $\liminf_{\lambda\uparrow\infty} (1/\lambda)\Lambda(\lambda) \geq m$, and hence taking $m\uparrow\esssup{\prob}{Y}$ gives $\lim_{\lambda\uparrow\infty} (1/\lambda)\Lambda(\lambda) = \esssup{\prob}{Y}$. A similar calculation shows that $\lim_{\lambda\downarrow-\infty}(1/\lambda)\Lambda(\lambda) = \essinf{\prob}{Y}$, proving both $(1)$ and $(2)$ above. That $\dot{\Lambda}(\lambda) = \espalt{}{Y e^{\lambda Y}}/\espalt{}{e^{\lambda Y}}$ for $\lambda\in\reals$ follows by Assumption \ref{ass: abstract_int}, the dominated convergence theorem and the inequality
\begin{equation}
|x|e^{\lambda x} \leq C(\lambda)\left(e^{2\lambda x} + e^{-2\lambda x}\right);\qquad x\in\reals\label{Eq:Inequality1},
\end{equation}
for some constant $C(\lambda)<\infty$. Now, by Jensen's inequality, $\espalt{}{e^{\lambda Y}} \geq e^{-|\lambda|\espalt{}{|Y|}}$ and Assumption \ref{ass: abstract_int} ensures that $\espalt{}{|Y|} < \infty$.  Furthermore, as for any $\lambda\in\R$ there exists some constant $C(\lambda)$ so that (\ref{Eq:Inequality1}) holds,
it follows again from Assumption \ref{ass: abstract_int} that $|\espalt{}{Ye^{\lambda Y}}| < \infty$, which yields $(3)$ finishing the proof.

\end{proof}

\begin{lemma}\label{lem: rf_coerciv}
Let Assumptions \ref{ass: asympt_sc} and \ref{ass: ldp} hold. For $\delta$ and $I$ as in Assumption \ref{ass: ldp}, $\liminf_{|y|\uparrow\infty} I(y)/|y| \geq \delta$.
\end{lemma}
\begin{proof}
In view of \eqref{eq: cgf_limit_int} and Varadhan's integral lemma for all $\eps \in (-\delta,\delta)$ it holds that
\begin{equation*}
\Gamma(\eps)=\lim_{n\uparrow\infty} \frac{1}{r_n}\log\left(\espalt{\prob^n}{e^{\eps r_n Y_n}}\right) = \sup_{y\in\R}\left(\eps y - I(y)\right) < \infty.
\end{equation*}
For $\eps > 0$ this gives for $y>0$ that $I(y)/y \geq \eps - \Gamma(\eps)/y$, from which the result follows by taking $y\uparrow\infty$ and $\eps\uparrow\delta$.  For $\eps < 0$ this gives for $y<0$ that $I(y)/(-y) \geq -\eps - \Gamma(\eps)/(-y)$, from which the result follows by taking $y\downarrow-\infty$ and $\eps\downarrow -\delta$.
\end{proof}

\begin{lemma}\label{lem: ldp_zero_fact}
Let Assumptions \ref{ass: asympt_sc} and \ref{ass: ldp} hold. Let $\delta$ be as in Assumption \ref{ass: ldp}. For any $\eps \in (-\delta,\delta)$ set
\begin{equation}\label{eq: l_u_eps}
l^\eps = \inf\cbra{y: y\in\argmax_{y\in\R}\left(\eps y - I(y)\right)};\qquad u^{\eps} = \sup\cbra{y:\ \argmax_{y\in\R}\left(\eps y - I(y)\right)}.
\end{equation}
Then $\lim_{\eps \downarrow 0} l^\eps  = 0 = \lim_{\eps\downarrow 0}u^{\eps}$.
\end{lemma}
\begin{proof}
This lemma should be known in the literature, but given that we could not find an exact reference, we provide a proof. As $I(y) =0 \Leftrightarrow y = 0$ we have $u^{\eps}\leq 0$ for $\eps < 0$, $0\leq l^{\eps}$ for $\eps > 0$ and $l^{\eps} = u^{\eps} = 0$ for $\eps = 0$.  Furthermore, by Lemma \ref{lem: rf_coerciv} we know for $|\eps| < \delta/2$ that $-K \leq l^\eps \leq u^{\eps} \leq K$ for some $K>0$ which does not depend upon $\eps$.  Now, let $\eps < 0, \eps\rightarrow 0$ and assume by way of contradiction that $l^{\eps}\rightarrow -l < 0$ for some $l>0$.  By definition of $l^{\eps}$ and this implies there exists a sequence $y_{\eps}\rightarrow y < -l/2$ such that for each $\eps$, $y_\eps\in\argmax_{y\in\R}\left(\eps y - I(y)\right)$. Therefore, we have that
\begin{equation*}
0 \leq \liminf_{\eps\downarrow 0} \left(\eps y_{\eps} - I(y_\eps)\right) \leq -I(y),
\end{equation*}
where the last inequality follows by the lower semi-continuity of $I$.  This gives  $I(y)\leq 0$ which by Assumption \ref{ass: ldp} is impossible as $y < -l/2$ and $I(y) =0$ if and only if $y=0$. Thus the result follows for $\eps < 0,\eps\rightarrow 0$ as we already know that $u^{\eps}\leq 0$.  A similar argument for $\eps > 0,\eps\rightarrow 0$ finishes the proof.
\end{proof}

\begin{lemma}\label{lem: exp_tilt_fact}
Let Assumptions \ref{ass: asympt_sc} and \ref{ass: ldp} hold.  Let $\delta$ be as in Assumption \ref{ass: ldp}. For $\eps\in (-\delta,\delta)$ define
\begin{equation*}
p^{\eps}_n = \frac{\espalt{\prob^n}{Y_n e^{\eps r_n Y_n}}}{\espalt{\prob^n}{e^{\eps r_n Y_n}}}.
\end{equation*}
Then
\begin{equation*}
0 = \liminf_{\eps\rightarrow 0}\liminf_{n\uparrow\infty} p^{\eps}_n = \limsup_{\eps\rightarrow 0}\limsup_{n\uparrow\infty} p^{\eps}_n.
\end{equation*}
\end{lemma}

\begin{proof}

Recall the function $\Lambda_n(\lambda) =  \log\left(\espalt{\prob^n}{e^{\lambda Y_n}}\right), \lambda\in\R$ from \eqref{eq: Y_n_cgf}. By Assumption \ref{ass: asympt_sc}, $\Lambda_n(\lambda)$ is strictly convex with $\dot{\Lambda}_n(\lambda) = \espalt{\prob^n}{Y_n e^{\lambda Y_n}}/\espalt{\prob^n}{e^{\lambda Y_n}}$.  Now, define the function $\Lambda(\lambda)= \sup_{y\in\R}\left(\lambda y - I(y)\right)$. Note that by construction, $\Lambda(\lambda)$ is convex. Furthermore, (\ref{eq: varad_pos}) states that $(1/r_n)\Lambda_n(\eps r_n) \rightarrow \Lambda(\eps)$ as $n\uparrow\infty$ for $\eps\in (-\delta,\delta)$.

By Proposition \ref{prop: opt_qn_no_n}, the map $\lambda\mapsto -\lambda p^{\eps}_n + \Lambda_n(\lambda)$ is minimized uniquely at $\lambda = \eps r_n$. This gives for all $\gamma\in\R$ that
\begin{equation}\label{eq: n_subdiff}
-\eps r_n p^{\eps}_n + \Lambda_n(\eps r_n) \leq -\gamma p^{\eps}_n + \Lambda_n(\gamma).
\end{equation}
Taking $\gamma$ to be $(\eps + \lambda)r_n$ where $\lambda>0$ is such that $\eps +\lambda < \delta$ it follows that
\begin{equation*}
-\eps r_n p^{\eps}_n + \Lambda_n(\eps r_n) \leq -(\eps+\lambda)r_n p^{\eps}_n + \Lambda_n((\eps+\lambda)r_n).
\end{equation*}
Canceling out the $-\eps p^{\eps}_n$ terms, dividing by $\lambda r_n$ and rearranging terms gives
\begin{equation*}
p^{\eps}_n \leq \frac{1}{\lambda}\left(\frac{1}{r_n}\Lambda_n\left((\eps+\lambda)r_n\right) - \frac{1}{r_n}\Lambda_n(\eps r_n)\right).
\end{equation*}
Taking $n\uparrow\infty$ gives
\begin{equation*}
\liminf_{n\uparrow\infty} p^{\eps}_n\leq \limsup_{n\uparrow\infty} p^\eps_n \leq \frac{1}{\lambda}\left(\Lambda(\eps+\lambda) - \Lambda(\eps)\right).
\end{equation*}
Taking $\lambda\downarrow 0$ gives
\begin{equation*}
\liminf_{n\uparrow\infty} p^{\eps}_n\leq \limsup_{n\uparrow\infty} p^\eps_n \leq \dot{\Lambda}_{+}(\eps),
\end{equation*}
where $\dot{\Lambda}_{+}(\eps)$ is the right derivative of $\Lambda$  at $\eps$. Similarly, coming back to \eqref{eq: n_subdiff}, taking $\gamma = (\eps-\lambda)r_n$ where $\lambda>0$ is such that $-\delta < -\eps -\gamma$ it follows that
\begin{equation*}
-\limsup_{n\uparrow\infty} p^{\eps}_n\leq -\liminf_{n\uparrow\infty} p^\eps_n \leq \dot{\Lambda}_{-}(\eps).
\end{equation*}
where $\dot{\Lambda}_{-}(\eps)$ is the left derivative of $\Lambda$  at $\eps$.
Therefore, by \cite[Theorem 23.2]{MR1451876} it follows that $\limsup_{n\uparrow\infty} p^{n}_{\eps},\liminf_{n\uparrow\infty} p^{n}_{\eps}  = \partial \Lambda(\eps)$. Now let $l^{\eps}\in\partial\Lambda(\eps)$. The claim is that $\lim_{\eps\rightarrow 0} |l^{\eps}| = 0$.  To see this, assume first by way of contradiction that there exits some $\tau>0$ such that $l^{\eps_k}\geq \tau$ for some sequence $\eps_k\rightarrow 0$.  Take $0 < \lambda < \delta$ such that $\lambda > \eps_k$ for all $k$ large enough.  By definition of the sub-differential it follows that
\begin{equation*}
\Lambda(\lambda) \geq \Lambda(\eps_k) + l^{\eps_k}(\lambda-\eps_k) \geq \Lambda(\eps_k) + \tau(\lambda - \eps_k).
\end{equation*}
For $\lambda$ small enough (still larger that $\eps_k$) Lemma \ref{lem: rf_coerciv} implies there exists some $y^{\lambda}\in\argmax_{y\in\R}(\lambda y - I(y))$ so that $\Lambda(\lambda) = \lambda y^{\lambda} - I(y^{\lambda})$. This implies
\begin{equation*}
\lambda y^{\lambda} - I(y^{\lambda}) \geq \Lambda(\eps_k) + \tau(\lambda-\eps_k).
\end{equation*}
As $\Lambda$ is convex, finite in $(-\delta,\delta)$  and $\Lambda(0) =0$, it follows by the continuity of $\Lambda$ in $(-\delta,\delta)$ that taking $k\uparrow\infty$ yields $\lambda y^{\lambda} - I(y^{\lambda}) \geq \tau\lambda$, or $y^{\lambda} \geq \tau + I(y^{\lambda})/\lambda \geq \tau$. Taking $\lambda\downarrow 0$ and using Lemma \ref{lem: ldp_zero_fact} gives that $0\geq \tau$, a contradiction. Similarly, assume by way of contradiction that there exits some $\tau>0$ such that $l^{\eps_k}\leq -\tau$ for some sequence $\eps_k\rightarrow 0$.  Take $-\delta < \lambda < 0$ such that $\lambda < \eps_k$.  By definition of the sub-differential it follows that
\begin{equation*}
\Lambda(\lambda) \geq \Lambda(\eps_k) + l^{\eps_k}(\lambda-\eps_k) \geq \Lambda(\eps_k) - \tau(\lambda - \eps_k).
\end{equation*}
For $\lambda$ small enough magnitude (though still less than $\eps_k$) Lemma \ref{lem: rf_coerciv} implies there exists some $y^{\lambda}\in\argmax_{y\in\R}(\lambda y - I(y))$ so that $\Lambda(\lambda) = \lambda y^{\lambda} - I(y^{\lambda})$. This implies
\begin{equation*}
\lambda y^{\lambda} - I(y^{\lambda}) \geq \Lambda(\eps_k) - \tau(\lambda-\eps_k).
\end{equation*}
As above, by taking $k\uparrow\infty$ one obtains $-y^{\lambda} \geq \tau + I(y^{\lambda})/(-\lambda) \geq \tau$. Taking $\lambda\downarrow 0$ and using Lemma \ref{lem: ldp_zero_fact} gives that $0\geq \tau$, a contradiction.  Thus, it follows that $|l^{\eps}|\rightarrow 0$ as $\eps \rightarrow 0$ for all $l^{\eps}\in \partial \Lambda(\eps)$ and hence the result follows.
\end{proof}

\bibliographystyle{siam}
\def\polhk#1{\setbox0=\hbox{#1}{\ooalign{\hidewidth
  \lower1.5ex\hbox{`}\hidewidth\crcr\unhbox0}}}


\def\polhk#1{\setbox0=\hbox{#1}{\ooalign{\hidewidth
  \lower1.5ex\hbox{`}\hidewidth\crcr\unhbox0}}}

\end{document}